\documentclass[11pt, a4paper]{amsart}
\usepackage{amsthm}
\usepackage[colorlinks=true,citecolor=cyan]{hyperref}
\usepackage{amsmath}
\usepackage{amsfonts}
\usepackage{amssymb}
\usepackage{stmaryrd}  
\usepackage{mathrsfs} 
\usepackage{esint} 
\usepackage{array} 
\usepackage{latexsym}
\usepackage{lmodern} 
\usepackage{graphicx} 
\usepackage{graphics} 
\usepackage{epsfig} 
\usepackage{color} 
\usepackage[dvipsnames]{xcolor}
\usepackage[all]{xy}
\usepackage[new]{old-arrows}
\usepackage{textcomp}
\usepackage{stmaryrd}
\DeclareGraphicsExtensions{.eps,.pdf,.jpeg,.png} 
\usepackage{fancyhdr}
\usepackage{multirow} 
\usepackage{comment}
\usepackage{pgf,tikz}
\usepackage{paralist}
\usepackage{faktor}
\usepackage{dirtytalk}
\usepackage{url}
\usepackage{blkarray,bigstrut}
\usepackage{nicematrix}
\usepackage{enumerate}

\makeatletter
\renewcommand*\env@matrix[1][*\c@MaxMatrixCols c]{%
  \hskip -\arraycolsep
  \let\@ifnextchar\new@ifnextchar
  \array{#1}}
\makeatother

\newtheorem{corollary}{Corollary}[section]
\newtheorem{proposition}{Proposition}[section]
\newtheorem{theorem}{Theorem}[section]
\newtheorem{lemma}{Lemma}[section]

\theoremstyle{definition}
\newtheorem{definition}{Definition}[section]
\newtheorem{remark}{Remark}[section]
\newtheorem{example}{Example}[section]
\newtheorem{notation}{Notation}[section]

\title{Shi variety corresponding to an affine Weyl group}

\author{Nathan Chapelier-Laget}

\begin{document}

\maketitle

\begin{abstract}
Let $W$ be an irreducible Weyl group and $W_a$ its affine Weyl group. In this article we show that there exists a bijection between $W_a$ and the integral points of an affine variety, denoted $\widehat{X}_{W_a}$, which we call the Shi variety of $W_a$. In order to do so, we use Jian-Yi Shi's characterization of alcoves in affine Weyl groups. We then study this variety further.  We highlight combinatorial properties of the irreducible components of $\widehat{X}_{W_a}$ and we show how they are related to a fundamental parallelepiped  $P_{\mathcal{H}}$. 
\end{abstract}

\bigskip
\bigskip

%{
  %\hypersetup{linkcolor=blue}
  \tableofcontents
%}

\section{Introduction}

One of the reasons that Coxeter groups are so well explored is due to their appearance as a fundamental structure in several mathematical theories, such as group theory, representation of algebras,  combinatorics, geometry, mathematical physics etc.

 The cell theory of a Coxeter group $W$, as defined by Kazhdan and Lusztig in \cite{kazhLus}, is related to many disparate and deep topics in mathematics, including the representation theory of $W$ and its Hecke algebra, the representation theory of algebraic groups, finite groups of Lie type, the geometry of unipotent conjugacy classes in simple complex algebraic groups, Lie algebras, primitive
ideals in universal enveloping algebras, singularities of Schubert varieties etc. 

One of the main challenge of this theory is to give a precise description of the cells of $W$ and to explore their properties. Describing the cells is in general a very difficult question since there is no general method. Many questions on this subject are still open.

When the Coxeter group is an affine Weyl group, this theory yielded a lot of research works in recent decades. It was first studied by G. Lusztig in \cite{LusCell} where he described all the cells of the affine Weyl groups of rank $\leq 2$.

In 1983, J.Y. Shi described all the cells of $W(\widetilde{A}_n)$ (see \cite{JYS4}). To do so he introduced the notion of sign types, which leads ,inter alia, to the notion of signed regions (the notion of signed region is  discussed in Sections \ref{connection Shi cell} and \ref{consequences}). These notions stem from a characterization of the elements of $W(\widetilde{A}_n)$ in terms of certain inequalities. A few years later he generalized in \cite{JYS1, JYS2} this work to any other affine Weyl group (see Theorem \ref{thJYS1}). We refer to this as \emph{Shi's characterization}.

 In this paper we revisit, extend and explore further Shi's characterization. We then introduce and investigate an affine variety $\widehat{X}_{W_a}$, called \emph{the Shi variety} of $W_a$, associated to an affine Weyl group $W_a$. 
 
 The main result of this article, which will be stated  in Theorem \ref{TH central}, can be reformulated as follows:
 
 \begin{theorem}\label{theo intro}
 We have a bijection between $W_a$ and the integral points of $\widehat{X}_{W_a}$, that is $$W_a \simeq \widehat{X}_{W_a}(\mathbb{Z}).$$
 \end{theorem}
 
\subsection{Notations and setup} For further developments we need to recall some basics about affine Weyl groups (all this material is explained in detail in Section \ref{section aff}).

Let $V$ be a Euclidean space, $\Phi \subset V$ be an irreducible crystallographic root system, and $W$ be the corresponding Weyl group. Associated to $\Phi$ there exists an infinite hyperplane arrangement, denoted $\mathcal{H}$, known as the affine Coxeter arrangement. This hyperplane arrangement cuts $V$ into simplices of the same volume which are called alcoves. The set of alcoves is denoted by $\mathcal{A}$.

The associated affine Weyl group $W_a$ acts regularly on $\mathcal{A}$, implying in particular that there exists a one-to-one correspondence between $W_a$ and $\mathcal{A}$ (see \cite{Hum} ch.4 or Section \ref{section aff} for more details). Let $A_w$ be the alcove  corresponding to $w \in W_a$. Then $A_w$ gives rise to a collection of integers $k(w,\alpha)$ for $\alpha \in \Phi^+$ (see Section \ref{shi para} for the definition of $k(w,\alpha)$).

 \subsection{Connection between the Shi variety and the cells}\label{connection Shi cell}
 This section can be read independently of the rest of this article, but for the reader who would be interested in a link with the cell theory, here it is.
 
  As stated in Theorem \ref{theo intro}, we realize the elements $w \in W_a$ as integral points $\iota(w) \in \mathbb{Z}^{|\Phi^+|}$ of $\widehat{X}_{W_a}$ (the map $\iota$ is introduced in Section \ref{phi-rep}). Then using J.Y. Shi's work \cite{JYS4},  we can determine in type $A$ which cell contains $w$  by looking at the sign of each coordinate of $\iota(w)$.
  More precisely, for an integer $k$ we define  ${sg(k) = 0}$ if $k=0$, $sg(k) = +$ if $k >0$ and $sg(k) = -$ if $k< 0$. For a $\Phi^+$-tuple of integers $(k_{\alpha})_{\alpha \in \Phi^+}$ we define the map $Sg$ as $Sg((k_{\alpha})_{\alpha \in \Phi^+})= (sg(k_{\alpha}))_{\alpha \in \Phi^+}$. We denote 
 $$
 Sg(w) := Sg((k(w,\alpha))_{\alpha \in \Phi^+}).
 $$ 
 
 A \emph{signed region} $\Gamma$ of  $W_a$ is by definition a subset of $W_a$ such that $Sg(w) =Sg(w')$ for all $w, w' \in \Gamma$. Then we can re-express the above statement, namely: if $Sg(w) =Sg(w')$ then $w$ and $w'$ are in the same cell. Moreover, J.Y. Shi established that each cell in type $A$ is a union of signed regions \cite{JYS4}.
 
 To connect the cells with the Shi variety we need to introduce the notion of generalized orthant. In the literature an orthant of $\mathbb{R}^m$ is a subset defined by constraining each Cartesian coordinate to be nonnegative or nonpositive, which gives $2^m$ such subspaces. We call generalized orthant a supspace of  $\mathbb{R}^m$ defined by constraining each Cartesian coordinate to be strictly positive or strictly negative or zero. Therefore, it is clear that any signed region of $W_a$ is characterized in $\widehat{X}_{W_a}$ by the intersection of $\widehat{X}_{W_a}(\mathbb{Z})$ and a generalized orthant. Thus, a cell is a union of intersections of $\widehat{X}_{W_a}(\mathbb{Z})$ with some generalized orthants.
 
 \subsection{Study of the Shi variety}
 After defining the Shi variety in Section \ref{decompo coeff}, we subsequently study the set of its irreducible components, which are affine subspaces. 
 
As mentioned above,  J.Y. Shi characterized in \cite{JYS1} any alcove $A_w$ by a $\Phi^+$-tuple of integers $(k(w,\alpha))_{\alpha \in \Phi^+} \in \mathbb{R}^{|\Phi^+|}$ subject to certain conditions (see Theorem \ref{thJYS1}). When we look at these $\Phi^+$-tuples, some linear relations appear between the coefficients. For example for $W(\widetilde{A}_2)$ embedded in $\mathbb{R}^3$,  the positions of the elements of length 4 and length 5 suggest that these elements lie on two parallel hyperplanes (see Figure \ref{points longueur 4 et 5}).

\begin{figure}[h!] \label{points longueur 4 et 5}
\centering
\includegraphics[scale=0.27]{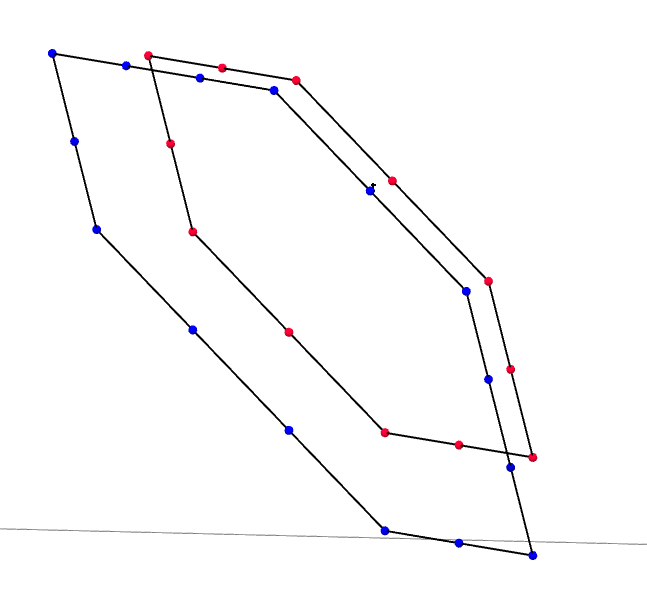} 
\caption{Elements of length 4 (red) and 5 (blue) in $W(\widetilde{A}_2)$.}
\end{figure}

 We show in this article that the previous observation is exactly what occurs, and in the above example, the equations of the hyperplanes are ${X_{13} = X_{12} + X_{23}}$ and $X_{13} = X_{12} + X_{23} +1$. \\

\noindent \textbf{The plan of this work.} The present paper has several goals that we describe now. First we recall in Section \ref{section aff} some standard definitions and terminologies related to affine Weyl groups.

$\bullet$ The first goal is to introduce the necessary tools in order to see elements of $W_a$ as isometries in  $\bigoplus_{\theta \in \Phi^+}\mathbb{R}e_{\theta}$. This leads us to define the $\Phi^+$-representation in Section \ref{phi-rep}. This new object has an important role in the understanding of the natural action of $W_a$ on the variety  $\widehat{X}_{W_a}$. To do so, we first generalize a few formulas of \cite{JYS1} about the coefficients $k(w,\alpha)$. This is done in Section \ref{phi-rep} and mainly in Proposition \ref{proposition k(sw,beta)}.

$\bullet$ The second goal, achieved in Section \ref{affine var}, is to build the variety $\widehat{X}_{W_a}$, whose integral points $\widehat{X}_{W_a}(\mathbb{Z})$ are in bijection with $W_a$, and to study the set of its irreducible components. The above statement is one of the main result of this article (see Theorem \ref{TH central}). To establish this, we first obtain a set of equations that cuts out an affine variety denoted $X_{W_a}$, and where the $\Phi^+$-tuples of integers of $W_a$ are solutions of these equations. However, this variety is too large and we only have an injection of $W_a$ into $X_{W_a}(\mathbb{Z})$. We then see how to shrink $X_{W_a}$ in order to get a one-to-one correspondence between the integers points of a subvariety and $W_a$. This is where we introduce the notion of admissible and admitted vectors (see Definitions \ref{admissible vector} and \ref{admitted vector}). These definitions yield a decomposition of $X_{W_a}$ as follows (see Theorem \ref{TH central}):
$$
 X_{W_a} := \bigsqcup\limits_{\lambda~\text{admissible}} X_{W_a}[\lambda],
 $$
whereas the subvariety $\widehat{X}_{W_a}$ inherits the decomposition:
$$
 \widehat{X}_{W_a} := \bigsqcup\limits_{\lambda~\text{admitted}} X_{W_a}[\lambda],
 $$

\noindent where the $X_{W_a}[\lambda]$ are the irreducible components of these two varieties. The difference between being admissible or admitted leads us to recover a polytope $P_{\mathcal{H}}$, which has  already been well studied (see \cite{BOURB} ch VI, $\S$ 1.10 or \cite{Hum} 4.9 or \cite{RichardKane} pages 131-134), and which plays a crucial role thereafter. Indeed, thanks to it, we give among others the number of irreducible components of $\widehat{X}_{W_a}$ (see Theorem \ref{TH central}).

\subsection{Some Consequences of the Shi variety}\label{consequences}

Besides revisiting Shi's characterization and extending it in terms of equations, the Shi variety encompasses several notions, both algebraic and combinatorial. 

 \textbf{-} As a first consequence, that we don't explain in this paper, the Shi regions (regions defined combinatorially in the Tits cone) are easy to construe and describe in terms of $\widehat{X}_{W_a}$ since they are naturally in bijection with the signed regions introduced in \cite{JYS2} and since the signed regions lie naturally in the Shi variety.  Furthermore, as mentioned previously, J.Y. Shi used in \cite{JYS4} the signed regions to describe the Kazhdan-Lusztig cells in type $A$. It is then natural to wonder what kind of results related to cells we might recover using this variety, first in type $A$, and then in other types ?

\textbf{-} The second consequence of the Shi variety, established in \cite{NC2}, is the relation between a specific conjugacy class of $W(A_n)$ and the set of irreducible components of $\widehat{X}_{W(\widetilde{A}_n)}$. More precisely we show that there is a natural bijection between the conjugacy class of $(1~2~\cdots~n+1)$, that we call the circular permutations, and the aforementioned set (see \cite{NC2} Theorem \textcolor{red}{1.2}). This result gives rise to several questions related to conjugacy class theory, such as: What can we say about other conjugacy classes in terms of a variant of the Shi variety (in any type) ?

\textbf{-} The paramaterization of the irreducible components of $\widehat{X}_{W_a}$ by the admitted vectors (see Theorem \ref{TH central}) gave rise to several interesting questions in type $A$. The set of admitted vector is partially ordered. Since this set is in bijection with the set of circular permutations, we obtain a partial order on this conjugacy class. In a joint work with A. Abram and C. Reutenauer, we study in \cite{NC3}  this poset with its different instances. We show in particular that this poset is a ranked poset of height ${n}\choose{3}$, showing us at the same time that it is cubic in $n$, which is rather unusual as we were told by R. Stanley. We then show that these two posets are isomorphic. From this isomorphism we derive several consequences, such as: it is a semidistributive lattice. We also connect these posets with Eulerian numbers, Young's lattice and triangulations of an $n$-gon together with their mutations. 

Finally, we introduce in \cite{NC3} a topological object, that we call  \emph{the line diagram}, and which allows us to describe precisely the bijection between our posets. The lines diagrams are connected with some knot theory transformations: the Reidemeister moves. This part is still an ongoing project.

\textbf{-} Using \cite{NC3, NC2} we established a mysterious connection between the two aforementioned posets and the poset of the right weak order of $W(A_n)$: the Hasse diagram of the poset of irreducible components of $\widehat{X}_{W(\widetilde{A}_n)}$ has the same number of edges as the poset of the right weak order of $W(A_n)$, even though these two posets are not isomorphic.

\textbf{-} In the light of the work done in \cite{NC3}, it was natural to investigate the combinatorial properties of this poset. In \cite{NC5} we study the poset of irreducible components of $\widehat{X}_{W_a}$ for any affine Weyl group and we show that it also has a structure of semidistributive lattice.

%\textbf{-} Darij Grinberg brought to our attention the problem 2 on the USAMO 2010 \cite{ECsite}. It turns out that the properties of the poset, studied in \cite{NC3}, give a direct answer to this problem.

\textbf{-} The poset of irreducible components of $\widehat{X}_{W_a}$ seems  to be related to another poset introduced by D. Speyer et al in \cite{GQS}. In this article the authors study the looping case of Mozes’s game of numbers, which concerns the orbits in the reflection representation of affine Weyl groups situated on the boundary of the Tits cone. Using an operation called \say{firing a vertex} on the set of vertices of a Dynkin diagram, they obtain a poset which is isomorphic to the dual of triangulation of the unit hypercube by Weyl chambers. Their hypercube is what we call in this article the fundamental  parallelepiped $P_{\mathcal{H}}$. It would be interesting  to connect these two posets, which seem to be same.

\textbf{-} In \cite{NC4} we study the cohomology in degree 1 of $W$. When we consider the irreducible components of $\widehat{X}_{W_a}$ that contain the generators of $W$, we get a natural framework in order to understand, from a  geometrical point of view, the group $H^1(W, \mathbb{Z}\Phi)$, but also the group $H^1(W, A\Phi)$ where $A$ is a ring satisfying $\mathbb{Z} \subset A \subset \mathbb{R}$. To do so, we connect the Cartan matrix of $\Phi$ with the equivalence relation corresponding to $H^1(W, \mathbb{Z}\Phi)$ when we identify $H^1(W, \mathbb{Z}\Phi)$ with the set of sections of a certain group extension naturally associated to $W$. We see then that the structure of $A$-module of $H^1(W, A\Phi)$ appears naturally in $\widehat{X}_{W_a}$.

\section{Affine Weyl groups and Shi parameterization} \label{section aff}

Let $V$ be a Euclidean space with inner product $(-, -)$. We denote $||x|| = \sqrt{(x,x)} $ for $x \in V$. Let $\Phi$ be an irreducible crystallographic root system in $V$. We assume here that $\Phi$ is essential, that is, taking the lattice ${\mathbb{Z} \Phi}$, one has ${\mathbb{Z}\Phi \otimes_{\mathbb{Z}}\mathbb{R}= V}$.  From now on, when we will say \say{root system} it will always mean irreducible essential crystallographic root system.

Let $W$ be the \emph{Weyl group} associated to $\mathbb{Z}\Phi$, that is the maximal (for inclusion) reflection subgroup of $O(V)$ admitting $\mathbb{Z}\Phi$ as a $W$-equivariant lattice.   

Let $\alpha \in \Phi$. We write 
$$
\begin{array}{ccccc}
s_{\alpha}  & : & V & \longrightarrow & V \\
                 &   & x & \longmapsto     & x-2\frac{( \alpha, x )}{(\alpha, \alpha )}\alpha.
\end{array}
$$

It is known that the Coxeter group associated to $\Phi$, i.e the subgroup of $O(V)$ generated by the reflections $s_{\alpha}$, is actually the Weyl group $W$. Each Weyl group has a structure of Coxeter group with set of Coxeter generators $S:=\{s_{\alpha_1},..,s_{\alpha_n}\}$.

Because of the classification of irreducible crystallographic root systems, we know that there are at most two possible root lengths in $\Phi$. We call short root the shorter ones. We require here that $||\alpha|| = 1$ for any short root $\alpha \in \Phi^+$. 

Let $\alpha \in \Phi$ such that  $\alpha = a_1\alpha_1 + \cdots + a_n\alpha_n$ with $a_i \in \mathbb{Z}$. The height of $\alpha$ (with respect to $\Delta$) is defined by the number $h(\alpha) = a_1 + \cdots+ a_n$. Height gives us an organizational principle for making inductive proofs. Height also provides a preorder on $\Phi^+$ defined as $\alpha \leq \beta$ if and only if $h(\alpha) \leq h(\beta)$. We also see that $h(\alpha + \beta) = h(\alpha) + h(\beta)$ and $h(-\alpha)=-h(\alpha)$ for all $\alpha, \beta, \alpha + \beta \in \Phi$. We denote by $-\alpha_0$ the \emph{highest short root} of $\Phi$.

\subsection{Affine Weyl groups}\label{affine Weyl groups}
From now on we will identify $\mathbb{Z}\Phi$ and the group of its associated translations. For $\alpha \in \Phi$ we write $\alpha^{\vee}:= \frac{2\alpha}{( \alpha, \alpha )}$. Let $k \in \mathbb{Z}$. Define the affine reflection as follows 

$$
\begin{array}{ccccc}
s_{\alpha,k}  & : & V & \longrightarrow & V \\
                 &   & x & \longmapsto     & x-(2\frac{( \alpha, x )}{( \alpha, \alpha )}-k)\alpha.
\end{array}
$$ 

We consider the subgroups $W_a$ and $W_a^{\vee}$ of Aff($V$) defined as follows
$$
W_a = \langle s_{\alpha,k}~|~\alpha \in \Phi, ~k \in \mathbb{Z \rangle}~~~~\text{and}~~~~W_a^{\vee} = \langle s_{\alpha^{\vee},k}~|~\alpha \in \Phi, ~k \in \mathbb{Z \rangle}.
$$

It is known that  $W_a \simeq  \mathbb{Z}\Phi\rtimes W $ (see \cite{Hum}, Ch 4). The group $W_a$ is called the \emph{affine Weyl group} associated to $\Phi$. The Coxeter group structure on $W$ induces a Coxeter group structure on $W_a$. The set $ S_a := \{s_{\alpha_1},\dots,s_{\alpha_n}\} \cup \{s_{-\alpha_0,1}\} $  is a set of Coxeter generators of $W_a$. For short we will write $S_a = \{ s_0, s_1, \dots s_n\}$ where $s_0 := s_{-\alpha_0,1}$ and $s_i = s_{\alpha_i}$ for $i =1,\dots, n$. Throughout this article, affine Weyl group means irreducible affine Weyl group.

Let us make the following comment. In the literature about Weyl groups and affine Weyl groups it is more common to define the affine Weyl groups (associated to the root system $\Phi$) with the reflections associated to the hyperplanes $\{ x \in V~|~(x,\alpha)=k\}$. Therefore what we would call the affine Weyl group associated to $\Phi$ would commonly be described in the literature as the affine Weyl group associated to $\Phi^\vee$, that is $W_a^\vee$. Our choice to call $W_a$ the affine Weyl group associated to $\Phi$ is because in \cite{JYS1} the author gave the definition of $W_a$ via the hyperplanes $\{ x \in V~|~(x,\alpha^{\vee})=k\}$, that is via the reflections $s_{\alpha,k}$, and since a lot of this article is based on Jian-Yi Shi's work it is natural to keep his conventions.

\subsection{Shi parameterization}\label{shi para}

 For any $\alpha \in \Phi$, any $k \in \mathbb{Z}$ and any $m \in \mathbb{R}$, we define the hyperplanes 
$$
H_{\alpha,k} = \{x \in V~|~s_{\alpha,k}(x)=x \} = \{ x \in V~|~ ( x, \alpha^{\vee} ) = k\},
$$
the half spaces 
$$
H_{\alpha,k}^{^{+}} = \{ x \in V|~ k < ( x,\alpha^{\vee} ) \}\text{~and~}
H_{\alpha,k}^{^{-}} = \{ x \in V|~ ( x,\alpha^{\vee}) < k \},
$$

\noindent and the strips
\begin{align*}
H_{\alpha,k}^m & = \{x \in V~|~k < ( x ,\alpha^{\vee} ) < k+m \}  = H_{\alpha,k}^{^{+}} \cap H_{\alpha,k+m}^{^{-}}.
\end{align*}

We denote by $\mathcal{H}$ the set of all the hyperplanes $H_{\alpha,k}$ with $\alpha \in \Phi^+$, $k \in \mathbb{Z}$, and by $\mathcal{H}^{\vee}$ the set of all the hyperplanes $H_{\alpha^{\vee},k}$ with $\alpha \in \Phi^+$, $k \in \mathbb{Z}$. We have the relation $H_{-\alpha,k} = H_{\alpha,-k}$. Therefore we need only to consider the hyperplanes $H_{\alpha,k}$ with $\alpha \in \Phi^+$ and $k \in \mathbb{Z}$. This last remark is quite convenient, because when we look at the collection of all the hyperplanes in $V$, we can think of them as being indexed either by elements of $\Phi$ or by elements of $\Phi^+$.

Furthermore, if we have two hyperplanes $H_{\alpha,k}$ and $H_{\alpha,k'}$  the geometric space strictly contained between  them is 
$$
H_{\alpha,\min{(k,k')}}^{|k-k'|}.
$$

The lattice $\mathbb{Z}\Phi$ acts naturally on the strips and this action is given by $\tau_x H_{\alpha,k}^m  = H_{\alpha, k+(\alpha^{\vee},x)}^{m}$. Moreover $W$ also acts on the strips as isometries of $\mathbb{R}^n$ and it follows that $W_a$ acts also on the set of strips.

 The connected components of 
$$
 V ~\backslash \bigcup\limits_{\tiny{\begin{subarray}{c}
 ~ ~\alpha \in \Phi^{+} \\ 
  k \in \mathbb{Z}
\end{subarray}}}
H_{\alpha,k} 
$$
are called \emph{alcoves}. We denote $\mathcal{A}$ the set of all the alcoves and $A_e$ the alcove defined as $A_e = \bigcap_{\alpha \in \Phi^+} H_{\alpha,0}^1$. Since $W_a$ acts on the strips it also acts on $\mathcal{A}$. It turns out that this action on $\mathcal{A}$ is regular (see \cite{Hum} chapter 4). Thus, there is a bijective correspondence between the elements of $W_a$ and all the alcoves. This bijection is defined by $w \mapsto A_w$ where $A_w := wA_e$. We call $A_w$ the corresponding alcove associated to $w \in W_a$. Any alcove of $V$ can be written as an intersection of particular strips, that is there exists a $\Phi^+$-tuple of integers $(k(w,\alpha))_{\alpha \in \Phi^+}$ such that 
$$
A_w = \bigcap\limits_{\alpha \in \Phi^+}H_{\alpha, k(w,\alpha)}^1.
$$

For any $w \in W_a$ and any $\alpha \in \Phi$ we use the convention$$k(w,-\alpha) = -k(w,\alpha).$$

 In the setting of affine Weyl groups the length of any element $w \in W_a$ is easy to compute via the coefficients $k(w,\alpha)$. Indeed, thanks to Proposition 4.3 in \cite{JYS1} we have $$\ell(w)=\sum\limits_{\alpha \in \Phi^+}|k(w,\alpha)|.$$

In \cite{JYS1} J.Y. Shi gave a characterization of the possible $\Phi^+$-tuples $(k_\alpha)_{\alpha \in \Phi^+}$ over $\mathbb{Z}$ such that these tuples are the tuples of elements in $W_a$.  In 1999 the same author gave an easier statement of this characterization. We give below these two characterizations.

\begin{theorem}[\cite{JYS1}, Theorem 5.2]\label{thJYS1} 
Let $A = \bigcap\limits_{\alpha \in \Phi^+} H^1_{\alpha,k_{\alpha}}$ with $k_{\alpha} \in \mathbb{Z}$. Then $A$ is an alcove, if and only if, for all $\alpha$, $\beta \in \Phi^+$ satisfying  $\alpha + \beta \in \Phi^+$, we have the following inequality

\begin{footnotesize}
\begin{equation}\label{Shi ineq}
||\alpha||^2k_{\alpha} + ||\beta||^2k_{\beta} +1 \leq ||\alpha + \beta||^{2}(k_{\alpha+\beta} +1) \leq ||\alpha||^2k_{\alpha} + ||\beta||^2k_{\beta} + ||\alpha||^2+ ||\beta||^2 + ||\alpha+\beta||^2 -1.
\end{equation}
\end{footnotesize}
\end{theorem}

\medskip

\begin{theorem}[\cite{JYS3}, Theorem 1.1] \label{thJYS3}
Let $A = \bigcap\limits_{\alpha \in \Phi^+} H^1_{\alpha,k_{\alpha}}$ with $k_{\alpha} \in \mathbb{Z}$. Then $A$ is an alcove, if and only if, for all $\alpha$, $\beta \in \Phi^+$ satisfying  $\gamma:= (\alpha^{\vee} + \beta^{\vee})^{\vee} \in \Phi^+$ the following inequality holds
\begin{equation}\label{Shi ineq simple}
k_{\alpha} + k_{\beta} \leq k_{\gamma} \leq k_{\alpha} + k_{\beta}  +1.
\end{equation}
\end{theorem}

\medskip

\begin{example} For $W_a = \widetilde{A}_2$, the positive root system of $A_2$ is  given by 3 roots, say ${\Phi^+ =\{\alpha, \beta, \alpha+\beta\}}$.  Shi's parameterization is shown in Figure \ref{dessin1}.

\begin{figure}[h!]
\centering
\includegraphics[scale=0.42]{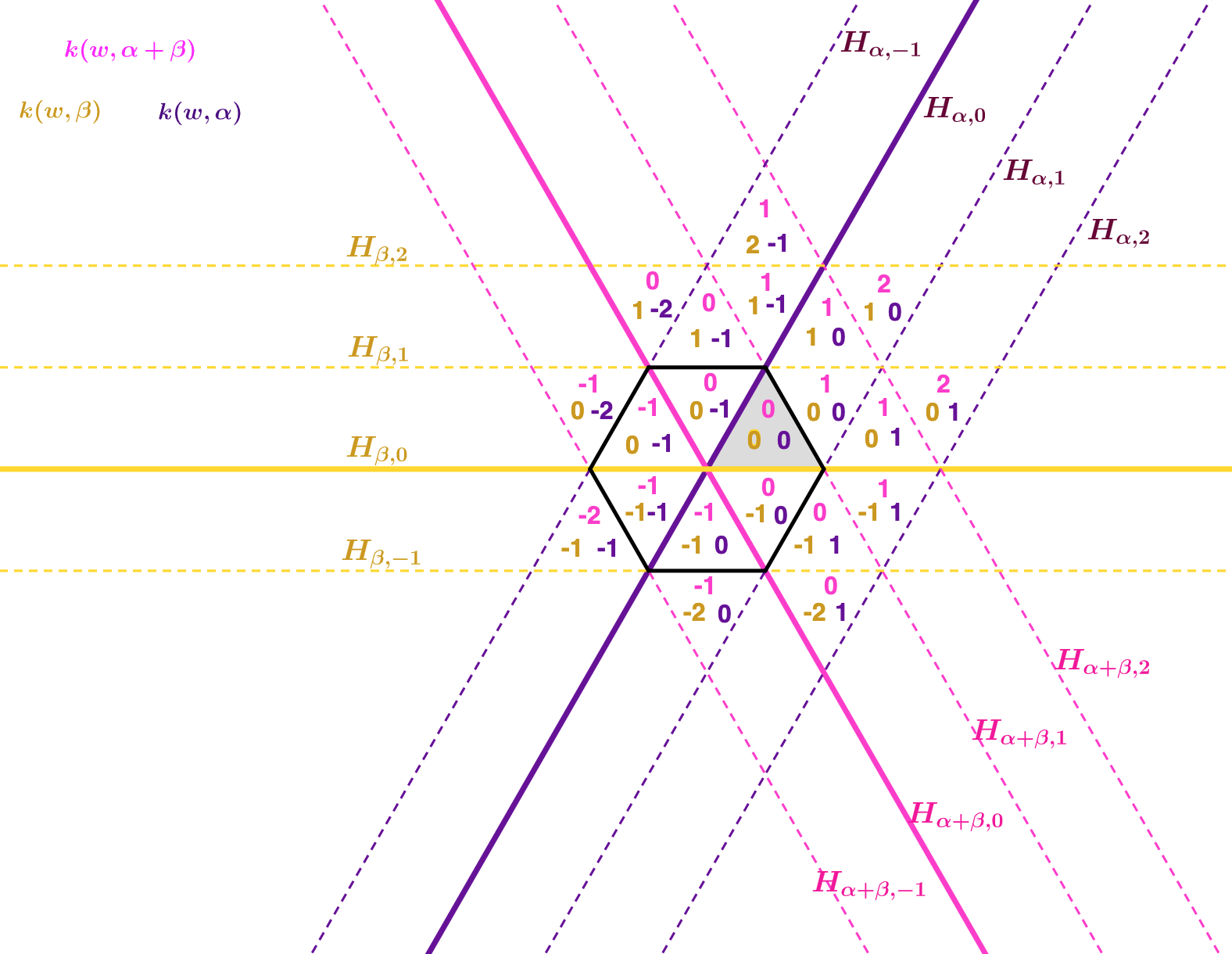} 
\caption{Shi parameterization of $W(\widetilde{A}_2)$.}
 \label{dessin1}
\end{figure}
\end{example}

\begin{remark}
1) Notice that it is a priori not easy for a $\Phi^+$-tuple of integers to satisfy the inequalities (\ref{Shi ineq simple}).

2) We also want to warn the reader that compared to the conventions of J.Y. Shi  in \cite{JYS1} and \cite{JYS2}, we swap left and right multiplication. The left multiplication is defined as follows: the alcove corresponding to $w'w$ is obtained by acting by the isometry $w'$ on the alcove $A_w$. The right multiplication is defined as follows: the alcove corresponding to $ws$ is obtained by crossing the wall of $A_w$ associated to the generator $s$. See for example Figure \ref{folding}.
\bigskip
\begin{figure}[h!]
\includegraphics[scale=0.43]{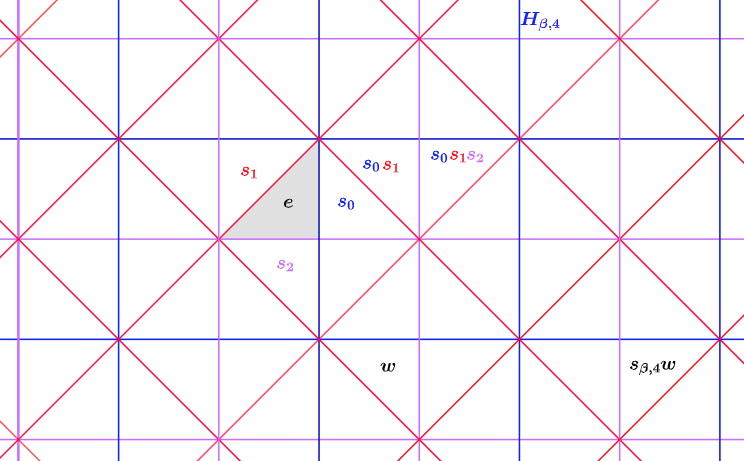} 
\caption{Folding of the generators of $W(\widetilde{B}_2)$.}
\label{folding}
\end{figure}
\end{remark}
\bigskip

\subsection{Fundamental parallelepiped $P_{\mathcal{H}}$}\label{P}

Let $\mathbb{Z}\Phi^{\vee}$ be the coroot lattice and let us write $\mathbb{Z}\Phi^{\vee} = \mathbb{Z}\alpha_1^{\vee} \oplus \cdots\oplus \mathbb{Z}\alpha_n^{\vee}$. We define its dual lattice $(\mathbb{Z}\Phi^{\vee})^{*}$ as 
$$
(\mathbb{Z}\Phi^{\vee})^{*} := \{ x \in V ~|~(x, y ) \in \mathbb{Z} ~\forall y \in \mathbb{Z}\Phi^{\vee} \}.
$$
The lattice $(\mathbb{Z}\Phi^{\vee})^{*}$ is called the \emph{weight lattice}. This lattice can be decomposed as ${(\mathbb{Z}\Phi^{\vee})^{*}= \mathbb{Z}\omega_1\oplus \cdots\oplus \mathbb{Z}\omega_n}$ where $\omega_i$ is such that $( \alpha_i^{\vee}, \omega_j ) = \delta_{ij}$. The elements $\omega_i$ are called the \emph{fundamental weights} (with respect to $\Delta$). Notice that it is also possible to define the weight lattice by 
$$
(\mathbb{Z}\Phi^{\vee})^* =  \{  x \in V ~|~( x, y ) \in \mathbb{Z} ~\forall y \in \Phi^{\vee} \}.
$$
Since $\Phi$ is crystallographic, $\Phi^{\vee}$ is also crystallographic and the inclusion $\mathbb{Z}\Phi^{\vee} \subset (\mathbb{Z}\Phi^{\vee})^{*}$ follows. Consequently, we are able to define the quotient group $\faktor{(\mathbb{Z}\Phi^{\vee})^{*}}{\mathbb{Z}\Phi^{\vee}}$. A good reference for these quotient groups is given in \cite{BOURB}. The \emph{index of connection} of $\Phi$ is by definition 
the cardinality of this quotient group. We denote it as 
$$
f_{\Phi} := \left | \faktor{(\mathbb{Z}\Phi^{\vee})^{*}}{\mathbb{Z}\Phi^{\vee}} \right |.
$$

We define $P_{\mathcal{H}} := \bigcap\limits_{\alpha \in \Delta}H_{\alpha,0}^1$ and $P_{\mathcal{H}^{\vee}} := \bigcap\limits_{\alpha \in \Delta}H_{\alpha^{\vee},0}^1$. The fundamental weights $\omega_i$ are some of the vertices  of $P_{\mathcal{H}}$ and we have $P_{\mathcal{H}} = \{\sum\limits_{i=1}^nc_i\omega_i~|~ c_i \in \llbracket 0,1 \rrbracket \}$. Since $(\omega_i, \omega_j ) \geq 0$ for all $i,j$, the element of maximal norm in $P_{\mathcal{H}}$  is the vertex $\rho :=\sum\limits_{i=1}^n\omega_i$. Moreover, if $z \in \text{cone}(\Delta)$ we have $(z,\omega_i) \geq 0$ for all fundamental weight $\omega_i$.

It is known (see \cite{BOURB} ch. VI, $\S$ 1, exercice 7) that the index of connection of $\Phi$ is the determinant of the Cartan matrix associated to $\Phi$. Moreover, the Cartan matrix associated to $\Phi^{\vee}$ is the transpose of the Cartan matrix associated to $\Phi$. It follows that the indexes of connection of $\Phi$ and $\Phi^{\vee}$ are the same. Finally, we define the two sets $$\text{Alc}(P_{\mathcal{H}}) := \{ w \in W_a~|~A_w \subset P_{\mathcal{H}} \}$$ and $$\text{Alc}(P_{\mathcal{H}^{\vee}}) := \{ w \in W_a^{\vee}~|~A_w \subset P_{\mathcal{H}^{\vee}} \}.$$ 

Finally, It is well known (see  \cite{RichardKane}, Section 11-6, Lemma C) that $$|\text{Alc}(P_{\mathcal{H}^{\vee}})|  = \frac{|W({\Delta})|}{f_{\Phi}}.$$

\begin{example} Let us take $W_a =W( \widetilde{B}_2)$ with simple system $\{\alpha_1, \alpha_2\}$. A short computation shows that $\omega_1 = \frac{1}{2}(2\alpha_1 + \alpha_2)$ and $\omega_2 = \alpha_1 + \alpha_2.$

 \vspace*{\stretch{1}}
\begin{figure}[h!]
\centering
\includegraphics[scale=0.5]{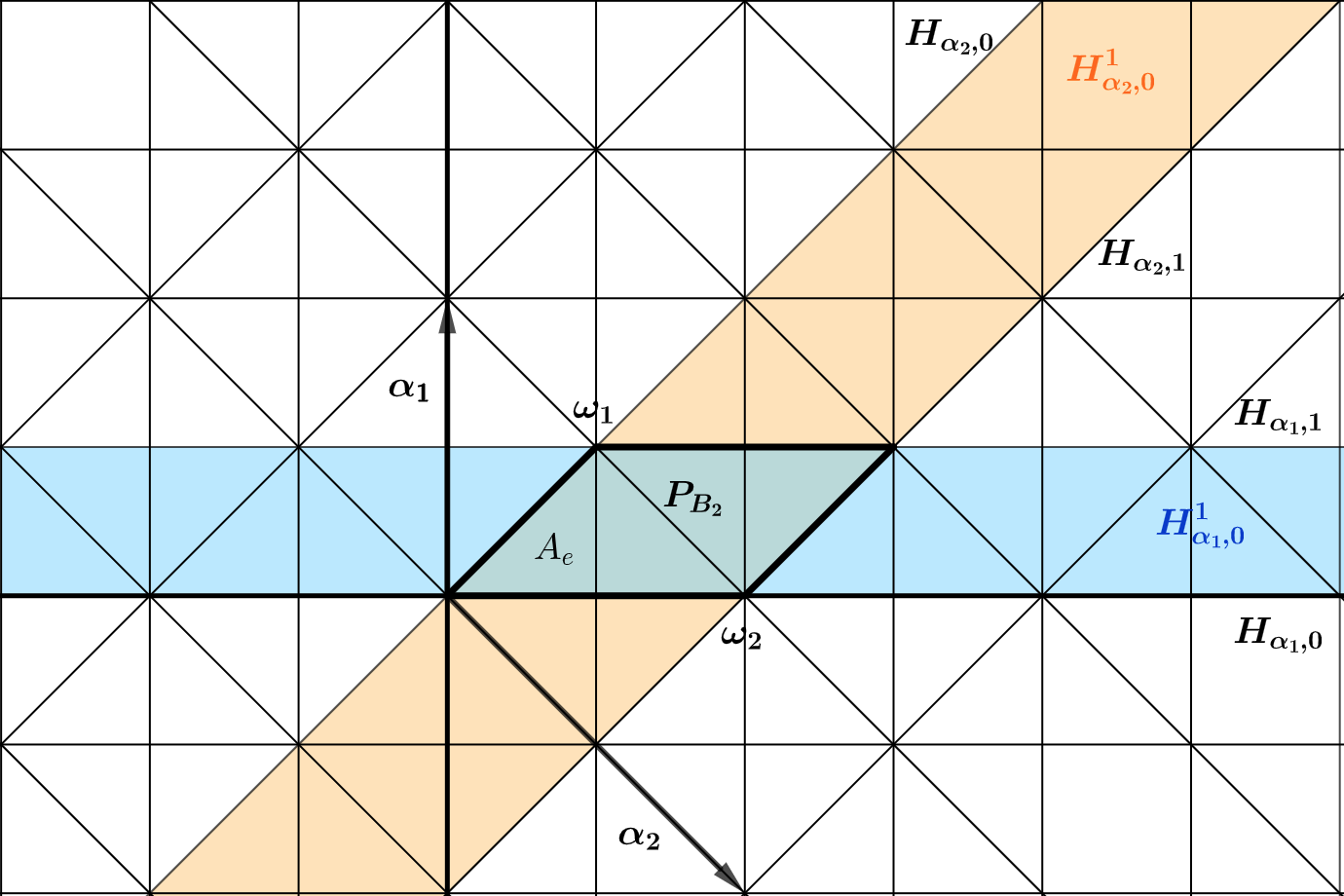} 
\caption{Fundamental parallelepiped $P_{B_2}$.}
\end{figure}
 \vspace*{\stretch{1}}

\end{example}

\textbf{~}\newpage

\section{$\Phi^+$-representation}\label{phi-rep}

The idea of this section is to see the elements of $W_a$ as affine isometries of a bigger space. Since any element $w \in W_a$ is characterized by a $\Phi^+$-tuple of integers $(k(w,\alpha))_{\alpha \in \Phi^+}$ and since $W_a$ acts geometrically on itself it is natural to want to see this action in $\mathbb{R}^{|\Phi^+|}$. For a Euclidean space $E$, we identify $E$ with the space of its translations, and we denote by Isom($E$) the space of its affine isometries, that is the elements of $E \rtimes GL(E)$ such that the Euclidean scalar product is preserved. In other words Isom($E$) $= E \rtimes O(E)$.

Let us write $\Phi^+=\{\beta_1, \beta_2, \dots,\beta_m\}$ with $m = |\Phi^+|$. We send $\Phi^+$ to the canonical basis of $\mathbb{R}^m$ via the map $\iota$. More specifically, since an element $w \in W_a$  is identified with the $\Phi^+$-tuple $(k(w,\gamma))_{\gamma \in \Phi^+}$ we set $\iota(w):=(k(w,\gamma))_{\gamma \in \Phi^+} $. The natural action of $s_{\alpha,k} \in $ Isom($\mathbb{R}^{^{|\Delta|}})$  corresponds to acting on $W_a$ by left multiplication, that is if $x \in A_w$ then $s_{\alpha,k}(x) \in A_{w'}$ where ${w' = s_{\alpha,k}w}$. We will denote the left multiplication by $w$ as $L_w$. The main goal of this section is to prove the following theorem:

\begin{theorem}\label{th Phi rep}
 There exists an injective morphism $F : W_a \rightarrow Isom(\mathbb{R}^m)$ such that for any $w \in W_a $ the following diagram commutes. This morphism is called the $\Phi^+$-representation of $W_a$.
 \begin{equation}\label{diag}
   \xymatrix{
 W_a \ar[r]^{L_w} \ar@{^{(}->}[d]_{\iota} &  W_a \ar@{^{(}->}[d]_{\iota} \\
    \mathbb{R}^m \ar[r]_{F(w)}                       & \mathbb{R}^m.
  }
 \end{equation}

\end{theorem}

We recall first of all a few results of \cite{JYS1} that we will use later. The following lemmas show how the coefficients $k(w,\alpha)$ behave when we act on $w$ by the available operations in $W_a$. 
We also want to point out that Proposition 4.2 of \cite{JYS1} has a typo, which was noticed by M. Dyer and which is fixed in Proposition \ref{descente indice k}.

\begin{lemma}[\cite{JYS1}, Lemma 3.1]  \label{k groupe fini}
Let $w$ be an element of $W\subset W_a$ and $\alpha \in \Phi^+$. Then 
$$
k(w,\alpha) =  \left\{
                          	\begin{array}{rl}
 						0      & \text{if}~~ w^{-1}(\alpha) \in \Phi^{+} \\
  				      	-1      & \text{if}~~ w^{-1}(\alpha) \in \Phi^{-}.
					    \end{array}
					    \right.
$$
\end{lemma}

\begin{lemma}[\cite{JYS1}, Lemma 3.2]\label{proposition translation}
Let $w$ be an element of $W$ and $x \in \mathbb{Z}\Phi$. Then for all $\alpha \in \Phi$ one has the following formula  $$k(\tau_xw, \alpha) = k(w,\alpha) + ( x, \alpha^{\vee} ).$$ 
\end{lemma}

\begin{corollary}[\cite{JYS1}, Formula (3.3.1)] \label{corollaire formule 3.3.1}
Let $\alpha \in \Phi$. As defined in Section \ref{section aff} we recall that $\alpha_0$ is the negative of the highest short root of $\Phi$ and $s_0:=s_{-\alpha_0,1}$. Then for all $0 \leq i \leq n$ we have  
$$
k(s_i,\alpha) =  \left\{
                          	\begin{array}{rl}
 						1     & \text{if}~~ \alpha = -\alpha_i \\
 						0    & \text{if}~~ \alpha \neq \pm \alpha_i \\
  				      	-1   & \text{if}~~  \alpha = \alpha_i. \\
					    \end{array}
					    \right.
$$
\end{corollary}

\begin{proposition}[\cite{JYS1}, Proposition 4.2]\label{descente indice k}
Let $w \in W_a$ and $s \in S_a$. Then  for all $\alpha \in \Phi^+$ one has the following formula   $$k(sw, \alpha) = k(w, \overline{s}(\alpha)) + k(s, \alpha).$$
\end{proposition}

\bigskip

\begin{proposition}\label{proposition k(sw,beta)}
Let $w \in W_a$, $\alpha, \beta \in \Phi^+ $ and $p \in \mathbb{Z}$. Then we have
\begin{itemize}
\item[1)]
 $
k(s_{\alpha}w,\beta) = 
\left\{
\begin{array}{rl}
   k(w,s_{\alpha}(\beta)) & \text{if}~~ s_{\alpha}(\beta) \in \Phi^{+} \\
   k(w,s_{\alpha}(\beta))-1 & \text{if}~~ s_{\alpha}(\beta) \in \Phi^{-} 
\end{array}
\right.
$
\item[2)]  $
k(s_{\alpha,p}w,\beta) = 
\left\{
\begin{array}{rl}
   k(w,s_{\alpha}(\beta)) - p( \alpha,  s_{\alpha}(\beta)^{\vee} )  & \text{if}~~ s_{\alpha}(\beta) \in \Phi^{+} \\
   k(w,s_{\alpha}(\beta))-1 -p( \alpha,  s_{\alpha}(\beta)^{\vee} ) & \text{if}~~ s_{\alpha}(\beta) \in \Phi^{-} 
\end{array}
\right.
$
\end{itemize}
\end{proposition}

\begin{proof}
1) Let $s_{\alpha}=s_1s_2\dots s_n$ be a reduced expression of $s_{\alpha}$. Since $s_{\alpha}$ is a reflection one has $s_1s_2\dots s_n = s_n\dots s_2s_1$. Using Proposition \ref{descente indice k} enough times we get that  ${k(s_{\alpha}w, \beta)  = k(w, s_{n}\dots s_1(\beta)) + \sum\limits_{i=1}^n k(s_i, s_{i-1}\dots s_1(\beta))}$. When $i=1$, the empty product on the righthand side is understood as the identity element.  Further, by using again Proposition \ref{descente indice k} enough times we get $k(s_{\alpha}, \beta) =  \sum\limits_{i=1}^n k(s_i, s_{i-1}\dots s_1(\beta))$. Thus we have 
\begin{align*}
k(s_{\alpha}w, \beta)  & = k(w, s_{n}\dots s_1(\beta)) + \sum\limits_{i=1}^n k(s_i, s_{i-1}\dots s_1(\beta)) \\ 
& = k(w, s_{\alpha}(\beta)) + k(s_{\alpha}, \beta).
\end{align*}

 Since ${s_{\alpha} \in W}$ and $ \beta \in \Phi^+$ we have by Lemma \ref{k groupe fini} that $k(s_{\alpha}, \beta)=0$ if $s_{\alpha}(\beta) \in \Phi^+$ and  $k(s_{\alpha}, \beta)=-1$ if $s_{\alpha}(\beta) \in \Phi^-$. The result follows.

2) This is just a consequence of 1) and Lemma  \ref{proposition translation}. Let us write $w=\tau_x\overline{w}$.  Notice first of all that $s_{\alpha}(\beta)^{\vee} = s_{\alpha}(\beta^{\vee})$ and 
\begin{align*}
k(s_{\alpha}w, \beta) = k(s_{\alpha}\tau_x\overline{w},\beta) &= k(\tau_{s_{\alpha}(x)}s_{\alpha}\overline{w},\beta) \\ &= k(s_{\alpha}\overline{w}, \beta) + ( s_{\alpha}(x), \beta^{\vee} ).
\end{align*}

Thus it follows that
\begin{align*}
k(s_{\alpha,p}w, \beta) & = k(\tau_{p\alpha}s_{\alpha}\tau_x\overline{w}, \beta) = k(\tau_{p\alpha+s_{\alpha}(x)}s_{\alpha}\overline{w}, \beta)   \\                              								   & = k(s_{\alpha}\overline{w}, \beta) + ( p\alpha +  s_{\alpha}(x), \beta^{\vee}) \\
    	                                  & = k(s_{\alpha}\overline{w}, \beta) + p( \alpha, \beta^{\vee}) + ( s_{\alpha}(x), \beta^{\vee} ) \\ 
								   &=k(s_{\alpha}\overline{w}, \beta)   - p( \alpha, s_{\alpha}(\beta^{\vee}) ) + (s_{\alpha}(x), \beta^{\vee}) \\
								   &= k(s_{\alpha}w, \beta) - p( \alpha, s_{\alpha}(\beta^{\vee}) ) \\
& = \left\{
\begin{array}{rl}
   k(w,s_{\alpha}(\beta)) - p( \alpha,  s_{\alpha}(\beta^{\vee}) )  & \text{if}~~ s_{\alpha}(\beta) \in \Phi^{+} \\
   k(w,s_{\alpha}(\beta))-1 -p( \alpha,  s_{\alpha}(\beta^{\vee}) ) & \text{if}~~ s_{\alpha}(\beta) \in \Phi^{-} 
\end{array}
\right. \\
& = \left\{
\begin{array}{rl}
   k(w,s_{\alpha}(\beta)) - p( \alpha,  s_{\alpha}(\beta)^{\vee} )  & \text{if}~~ s_{\alpha}(\beta) \in \Phi^{+} \\
   k(w,s_{\alpha}(\beta))-1 -p( \alpha,  s_{\alpha}(\beta)^{\vee} ) & \text{if}~~ s_{\alpha}(\beta) \in \Phi^{-} 
\end{array}
\right.
\end{align*}
\end{proof}

\begin{definition} Let $\alpha \in \Phi^+$ and $p \in \mathbb{Z}$.
\begin{itemize}
\item[1)]We define $L_{\alpha} \in GL_m(\mathbb{R})$ via the matrix $(\ell_{i,j}(\alpha))_{i,j \in \llbracket 1,m \rrbracket}$ where

$$
 \ell_{\beta_j, \beta_i}(\alpha) :=\ell_{j,i}(\alpha) = 
 \left\{
 \begin{array}{rl}
     1  &  \text{if} ~~ s_{\alpha}(\beta_i) = \beta_j															\\
    0  &  \text{if} ~~ s_{\alpha}(\beta_i) \neq \pm \beta_j												\\
   -1 &  \text{if} ~~ s_{\alpha}(\beta_i) = -\beta_j    													\\
 \end{array}
\right.
$$

\item[2)] We define the vector $v_{p,\alpha} \in \mathbb{R}^m$ as $v_{p,\alpha}=(v_{p,\alpha}(\gamma))_{\gamma \in \Phi^+}$ where 
$$
v_{p,\alpha}(\gamma) := 
 \left\{
 \begin{array}{rl}
    -p( \alpha , s_{\alpha}(\gamma)^{\vee}) &  \text{if} ~~ s_{\alpha}(\gamma) \in \Phi^+  					\\
        -1-p(  \alpha, s_{\alpha}(\gamma)^{\vee} ) &  \text{if} ~~ s_{\alpha}(\gamma) \in \Phi^-     			\\
 \end{array}
\right.
.
$$
For $\gamma \in \Phi^+$ we set the convention 
 $
v_{p,\alpha}(-\gamma):=-v_{p,\alpha}(\gamma).
$
\item[3)] We define $F(s_{\alpha,p})$ as the affine map such that for all $x \in \mathbb{R}^m$ one has 
 $$
 F(s_{\alpha,p})(x) := L_{\alpha}(x) + v_{p,\alpha}.
 $$
For short we will write $L_{\alpha}(w)$ instead of $L_{\alpha}(\iota(w))$. We denote $L_{\alpha}(w)[\theta]$ the coefficient in position $\theta$ of the vector $L_{\alpha}(w)$. This coefficient is called the $\theta$-coefficient of $L_{\alpha}(w)$. Similarly, the coefficient in position $\theta$ of $F(s_{\alpha,p})(w)$ is denoted $F(s_{\alpha,p})(w)[\theta]$ and we have  $F(s_{\alpha,p})(w)[\theta] = L_{\alpha}(w)[\theta] + v_{p,\alpha}(\theta)$. We  call it the $\theta$-coefficient of $F(s_{\alpha,p})(w)$.
\end{itemize}
\end{definition}

\begin{proposition}\label{proposition commutativity}
Let $\alpha, \theta \in \Phi^+$ and $p \in \mathbb{Z}$. Then 
\begin{itemize}
\item[1)]  $L_{\alpha}^2=id$ and $F(s_{\alpha,p}) \in$ Isom($\mathbb{R}^m)$.
\item[2)] $ L_{\alpha}(w)[\theta] = k(w,s_{\alpha}(\theta))$.
\item[3)]  The following diagram commutes
$$
  \xymatrix{
 W_a \ar[r]^{L_{s_{\alpha,p}}} \ar@{^{(}->}[d]_{\iota} &  W_a \ar@{^{(}->}[d]_{\iota} \\
    \mathbb{R}^m \ar[r]_{F(s_{\alpha,p})}                       & \mathbb{R}^m.
  }
$$
\end{itemize}
\end{proposition}

\begin{proof}
1) Since $s_{\alpha}$ is an involution, the way we defined $L_{\alpha}$ shows that $\ell_{i,j}(\alpha) = \ell_{j,i}(\alpha)$, that is $L_{\alpha}=L_{\alpha}^t$. Moreover each line and each column has only one $\ell_{i,j}(\alpha) \neq 0$. Then it follows that $L_{\alpha}^2=id$ and for instance we have $L_{\alpha} \in GL(\mathbb{R}^m)$. In order to show that $F(s_{\alpha,p}) \in$ Isom$(\mathbb{R}^m)$ we just have to show that $L_{\alpha} \in O(\mathbb{R}^m)$. But $L_{\alpha}^t=L_{\alpha}=L_{\alpha}^{-1}$ which means that $L_{\alpha} \in O(\mathbb{R}^m)$.

2) From the way we defined $L_{\alpha}$ the $\theta$-coefficient of $L_{\alpha}(w)$ has the value $$\sum\limits_{\gamma \in \Phi^+}\ell_{\theta,\gamma}(\alpha)k(w,\gamma).$$ Furthermore, we know that $\ell_{\theta, \gamma}(\alpha) \neq 0$ if and only if ${\gamma = \pm s_{\alpha}(\theta)}$. Therefore, if ${\gamma=s_{\alpha}(\theta)}$ we have  $s_{\alpha}(\theta) \in \Phi^+$  and it follows  $ \ell_{\theta, \gamma}(\alpha)=1$, which implies that the $\theta$-coefficient of $L_{\alpha}(w)$ is $\ell_{\theta, \gamma}(\alpha)k(w, \gamma) =k(w, s_{\alpha}(\theta))$. If now we have  $s_{\alpha}(\theta) = -\gamma$ then $s_{\alpha}(\theta) \in \Phi^-$, which implies that ${\ell_{\theta, \gamma}(\alpha) =-1}$. Thus, the $\theta$-coefficient of $L_{\alpha}(w)$ is  
$\ell_{\theta, \gamma}(\alpha)k(w, \gamma)=(-1)k(w, -s_{\alpha}(\theta)) =k(w,s_{\alpha}(\theta))$. Hence, in all cases we have the $\theta$-coefficient of $L_{\alpha}(w)$ equal to $k(w,s_{\alpha}(\theta))$.

3) Let $\theta \in \Phi^+$. Let us show that the $\theta$-coefficient of the vector $F(s_{\alpha,p})(w)$ is the same as the $\theta$-coefficient of the vector $\iota(s_{\alpha,p}w)$. By point 2) before we know that $L_{\alpha}(w)[\theta] = k(w,s_{\alpha}(\theta))$. Then we have 
$$
F(s_{\alpha,p})(w)[\theta] = L_{\alpha}(w)[\theta] + v_{p,\alpha}(\theta)= k(w,s_{\alpha}(\theta)) + v_{p,\alpha}(\theta).
$$
 With respect to $\iota(s_{\alpha,p}w)$, via the way we defined $\iota$ we know that the $\theta$-coefficient of $\iota(s_{\alpha,p}w)$ is $k(s_{\alpha,p}w, \theta)$. However, because of Proposition \ref{proposition k(sw,beta)} we know that 
\begin{align*}
k(s_{\alpha,p}w,\theta) & = 
\left\{
\begin{array}{rl}
   k(w,s_{\alpha}(\theta)) - p(\alpha,  s_{\alpha}(\theta)^{\vee} )  & \text{if}~~ s_{\alpha}(\theta) \in \Phi^{+} \\
   k(w,s_{\alpha}(\theta))-1 -p(\alpha,  s_{\alpha}(\theta)^{\vee} ) & \text{if}~~ s_{\alpha}(\theta) \in \Phi^{-} 
\end{array}
\right. \\
& = k(w,s_{\alpha}(\theta)) + \left\{
\begin{array}{rl}
   - p( \alpha,  s_{\alpha}(\theta)^{\vee} )  & \text{if}~~ s_{\alpha}(\theta) \in \Phi^{+} \\
   -1 -p( \alpha,  s_{\alpha}(\theta)^{\vee})  & \text{if}~~ s_{\alpha}(\theta) \in \Phi^{-} 
\end{array}
\right. \\
& = k(w,s_{\alpha}(\theta)) + v_{p,\alpha}(\theta).
\end{align*}

\end{proof}

\begin{proof}[Proof of Theorem \ref{th Phi rep}]
Let $w \in W_a$. Since $W_a = \langle s_{\alpha,p}, \alpha \in \Phi^+, p \in \mathbb{Z} \rangle$ we can write ${w =s_{\alpha_1,p_1}\dots s_{\alpha_q,p_q}}$ with $\alpha_i \in \Phi^+$ and $p_i \in \mathbb{Z}$. We now want to extend to $W_a$  the definition of $F$ by $F(w) = F(s_{\alpha_1,p_1})\dots F(s_{\alpha_q,p_q})$.

 However we need to check that this definition doesn't depend on the choice of the expression of $w$. Let us assume that we have two affine maps $\varphi_1(w),\varphi_2(w) \in \text{Aff}( \mathbb{R}^m)$ making the diagram (\ref{diag}) commute. We have in particular that
$$
\varphi_1(w)(\iota(x)) = (\iota \circ L_w)(x) =  \varphi_2(w)(\iota(x))
$$
 for all $x$ running over a maximal chain of the right weak order of $W$. Furthermore, Lemma \ref{k groupe fini} implies that the set
 $$
 \{\iota(x), x~ \text{running over a maximal chain of the right weak order of}~ W\}$$ 
 is an affine basis of $\mathbb{R}^m$. It follows that $\varphi_1(w) = \varphi_2(w)$ and then there is at most one affine map making the diagram (\ref{diag}) commute. Using Proposition \ref{proposition commutativity} 3) we see that $F(s_{\alpha_1,p_1})\dots F(s_{\alpha_q,p_q})$ is such an affine map. Thus $F(w)$ is well defined and makes the diagram (\ref{diag}) commute.

Using Proposition \ref{proposition commutativity} 1) it is clear that $F(w) \in \text{Isom}(\mathbb{R}^m)$, and by definition of $F$ it is also clear that the next diagram commutes for all $w_1, w_2 \in W_a$           
$$
  \xymatrix{
 W_a \ar[r]^{L_{w_1}} \ar@{^{(}->}[d]_{\iota} \ar@/^2pc/[rr]^{L_{w_2w_1}} &  W_a \ar@{^{(}->}[d]_{\iota} \ar[r]^{L_{w_2}} & W_a \ar@{^{(}->}[d]_{\iota}\\
    \mathbb{R}^m \ar[r]_{F(w_1)}    \ar@/_2pc/[rr]_{F(w_2w_1)}                   & \mathbb{R}^m                \ar[r]_{F(w_2)}                & \mathbb{R}^m .
  }
$$      

Thus $F$ is a morphism. This morphism is injective because $F(w_1) = F(w_2)$ implies for instance that  $F(w_1)(\iota(e)) = F(w_2)(\iota(e))$ which means that $\iota \circ L_{w_1}(e) = \iota \circ L_{w_2}(e)$, that is $\iota(w_1) = \iota(w_2)$. Since $\iota(w_1)=(k(w_1,\alpha))_{\alpha \in \Phi^+}$ and $ \iota(w_2) = (k(w_2,\alpha))_{\alpha \in \Phi^+}$ it follows that $w_1=w_2$.
   \end{proof}

\section{Structure of affine variety on $W_a$}\label{affine var}

Let $W_a$ be an affine Weyl group, $\Phi$ its crystallographic root system with $\Phi^+=\{\beta_1,\dots ,\beta_m \}$, and $\Delta = \{\alpha_1,\dots ,\alpha_n\}$ its simple system. The aim of this section is to understand, from a new geometric point of view, affine Coxeter groups based on their $\Phi^+$-tuple of integers obtained in \cite{JYS1}. Shi gave in \cite{JYS1} a characterization of elements in $W_a$ via a collection of inequalities. He also gave an easier characterization of the $\Phi^+$-tuples in Theorem \ref{Shi ineq simple}. The goal here is to transform these inequalities into equations.  From now on we will denote $A[X_{\Delta}] := A[X_{\alpha_1},\dots ,X_{\alpha_n}]$ and $A[X_{\Phi^+}] := A[X_{\beta_1},\dots ,X_{\beta_m}]$ where $A$ is a commutative ring. In particular we have $ A[X_{\alpha_1},\dots ,X_{\alpha_n}] \subset A[X_{\beta_1},\dots ,X_{\beta_m}] $.
For $w \in W_a$ and $Q \in A[X_{\Delta}]$ we denote by
$$
Q(w):=Q(k(w,\alpha_1),\dots ,k(w,\alpha_n)).
$$

\subsection{Decomposition of the coefficients $k(w,\alpha)$}\label{decompo coeff} In this section we give the main tool in order to construct the Shi variety.

\begin{theorem}\label{polynome}
Let $w \in W_a$. Then we have
\begin{itemize}
\item[1)] For all $\theta \in \Phi^+-\Delta$ there exists a linear polynomial $P_{\theta} \in \mathbb{Z}[X_{\Delta}]$  with positive coefficients and  $\lambda_{\theta}(w) \in \llbracket 0, h(\theta^{\vee})-1\rrbracket$ such that 
\begin{equation} \label{eq:P}
k(w,\theta) = P_{\theta}(w) + \lambda_{\theta}(w).
\end{equation}

\item[2)]  $ P_{\theta} = P_{\alpha} + P_{\beta}$ for all $\alpha, \beta \in \Phi^+$ such that ${\theta^{\vee} = \alpha^{\vee} + \beta
^{\vee}}$.
\end{itemize}
\end{theorem}

\begin{proof}
If $\theta \in \Delta$ then we set $P_{\theta} = X_{\theta}$ and  $\lambda_{\theta}(w) = 0$. Assume now that $\theta \in \Phi^+\setminus \Delta$.
We proceed by induction on the height $h$ of coroots. We have ${\theta^{\vee} \in (\Phi^{\vee})^+}$ and then there exist $ \alpha, \beta \in \Phi^+$ such that $\theta^{\vee} = \alpha^{\vee} + \beta^{\vee}$.
 Notice that it is always possible to write an element of $(\Phi^{\vee})^+\setminus\Delta^{\vee}$ as a sum of two others elements of $(\Phi^{\vee})^+$.  Thus we have $ \theta=(\alpha^{\vee} + \beta^{\vee})^{\vee} \in \Phi^+$ and because of Theorem \ref{thJYS3} it follows $k(w,\alpha) + k(w,\beta) \leq k(w,\theta) \leq k(w,\alpha) + k(w,\beta)+1$. Hence  $k(w,\theta) = k(w,\alpha) + k(w,\beta) + \gamma_{\theta}(w)$ where $\gamma_{\theta}(w) \in \llbracket 0,1 \rrbracket$.

If $h(\theta^{\vee}) = 2$ one has necessarily $\alpha^{\vee}, \beta^{\vee} \in \Delta^{\vee}$. Therefore by setting $P_{\theta} := X_{\alpha} + X_{\beta}$ and $\lambda_{\theta}(w):= \gamma_{\theta}(w)$ we have shown the base case.

Assume now that $h(\theta^{\vee}) = d + 1$ and assume that the formula (\ref{eq:P}) is true for all coroots of height less than $d+1$. We also assume that for all $\gamma \in \Phi^+$ satisfying $h(\gamma^{\vee}) < d+1$ the polynomial $P_{\gamma}$ is linear and its coefficients are positive. Since $\alpha^{\vee} + \beta^{\vee} = \theta^{\vee}$, we have $h(\alpha^{\vee}) < h(\theta^{\vee})$ and $h(\beta^{\vee}) < h(\theta^{\vee})$. It follows that ${k(w,\alpha) = P_{\alpha}(w) + \lambda_{\alpha}(w)}$  and $k(w,\beta) = P_{\beta}(w) + \lambda_{\beta}(w)$ with $P_{\alpha}, P_{\beta} \in \mathbb{Z}[X_{\Delta}]$ both linear, $\lambda_{\alpha}(w) \in \llbracket 0,h(\alpha^{\vee})-1 \rrbracket$  and $\lambda_{\beta}(w) \in \llbracket 0, h(\beta^{\vee})-1 \rrbracket$. It follows that:
\begin{equation*}
k(w,\theta) = P_{\alpha}(w) + \lambda_{\alpha}(w)+ P_{\beta}(w) + \lambda_{\beta}(w) + \gamma_{\theta}(w).
\end{equation*}

By setting $P_{\theta} := P_{\alpha} + P_{\beta}$ and $\lambda_{\theta}(w) = \lambda_{\alpha}(w)+ \lambda_{\beta}(w) + \gamma_{\theta}(w)$ we have $P_{\theta}$ linear, its coefficients are positive since those of $P_{\alpha}$ and $P_{\beta}$ are positive, and finally since ${h(\theta^{\vee}) = h(\alpha^{\vee}) + h(\beta^{\vee})}$ it results that
${\lambda_{\theta}(w) \in \llbracket 0, h(\theta^{\vee})-1\rrbracket}$. This ends the induction and the proof of points 1) and 2).
\end{proof}

\begin{example}\label{coeffB} Let $\Delta(B_n) := \{\alpha_1,\dots ,\alpha_n\}$ with its Dynkin diagram:
\begin{figure}[h!]
\centering
\includegraphics[scale=0.5]{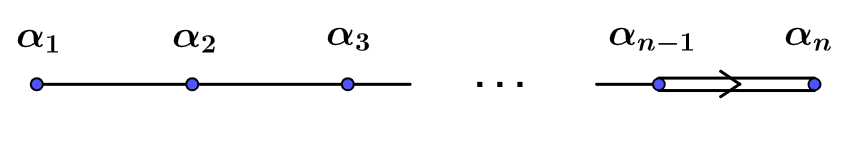} 
\caption{Dynkin diagram of type $B_n$.}
\label{Dyn Bn}
\end{figure}

It is well-known (see, for example, \cite{BOURB}) that its positive root system decomposes as 
$$
 \Phi^+(B_n) = I_1 \sqcup I_2 \sqcup I_3,
$$
 where ${I_1:= \{\sum\limits_{t=i}^{j-1} \alpha_t~ |~1 \leq i<j \leq n \}}$, ${I_2 :=\{\sum\limits_{t=i}^n \alpha_t ~|~1 \leq i \leq n \}}$ and $$I_3 := \{\sum\limits_{t=i}^{j-1} \alpha_t + 2 \sum\limits_{t=j}^n \alpha_t ~| ~ 1 \leq i <j \leq n \}.$$ It turns out that the short roots are those in $I_2$ and the long ones are those in $I_1 \sqcup I_3$. Moreover, the only simple root that is short is $\alpha_n$. Let $\alpha \in \Phi^+(B_n)$ and $w \in W_a$. After computation we get:
\begin{itemize}
\item[i)] If $\alpha = \sum\limits_{t=i}^{j-1}\alpha_t \in I_1$ then there exists $\lambda_{\alpha}(w) \in \llbracket 0, j-i-1 \rrbracket$ such that 
$k(w,\alpha) = \sum\limits_{t=i}^{j-1}k(w,\alpha_t) + \lambda_{\alpha}(w)$. Here $P_{\alpha} = \sum\limits_{k=i}^{j-1}X_{\alpha_k}$ and it is easy to see that $j-i =h(\alpha) = h(\alpha^{\vee})$.

\item[ii)]
 If $\alpha = \sum\limits_{t=i}^{n}\alpha_t  \in I_2$ then there exists $\lambda_{\alpha}(w) \in \llbracket 0, 2(n-i)\rrbracket$ such that $k(w,\alpha) = 2\sum\limits_{t=i}^{n-1}k(w,\alpha_t) + k(w,\alpha_n) + \lambda_{\alpha}(w)$. Here $P_{\alpha} :=2\sum\limits_{k=i}^{n-1}X_{\alpha_k} + X_{\alpha_n}$ and it is easy to see that $h(\alpha^{\vee}) = 2(n-i)+1$.

\item[iii)] If $\alpha = \sum\limits_{t=i}^{j-1}\alpha_t + 2 \sum\limits_{t=j}^n \alpha_t \in I_3$ then there exists $\lambda_{\alpha}(w) \in \llbracket 0, 2n-(i+j)\rrbracket$ such that 
$$
  k(w,\alpha)  = 
\left\{
\begin{array}{ll}
 \sum\limits_{t=i}^{j-1}k(w,\alpha_t) +  2\sum\limits_{t=j}^{n-1}k(w,\alpha_t) + k(w,\alpha_n) + \gamma_{\alpha}(w)& \text{if}~~ j<n \\
  \sum\limits_{t=i}^{n}k(w,\alpha_t)  + \gamma_{\alpha}(w)& \text{if}~~j=n.
\end{array}
\right. \\
$$

\noindent Here we have 
$$
P_{\alpha}  = 
\left\{
\begin{array}{ll}
 \sum\limits_{k=i}^{j-1}X_{\alpha_k} +  2\sum\limits_{k=j}^{n-1}X_{\alpha_k}+ X_{\alpha_n}& \text{if}~~ j<n \\
  \sum\limits_{k=i}^{n}X_{\alpha_k}  & \text{if}~~j=n,
\end{array}
\right. \\
$$
and it is easy to see that $h(\alpha^{\vee}) = 2n-(i+j) +1$.
\end{itemize}
\end{example}

\subsection{Definition of the Shi variety} Before giving the main definitions of this section we want to discuss some terminologies and notations, particularly to clarify what we mean by \say{affine variety} in this article. During the 20th century the notion of (algebraic) variety evolved and many sophisticated and deep notions can be involved. For some people an affine variety $V$ over a field $k$ is just the set of solutions of polynomial equations with coefficients in $k$, which is sometimes called an algebraic set. For others it is the ringed space $(V, \mathcal{O}_V)$ where $\mathcal{O}_V$ is the sheaf of regular functions over $V$, or again for others it is a ringed space $(X, \mathcal{O}_X)$ isomorphic to $(V, \mathcal{O}_V)$. 

In this article the term \emph{affine variety} is employed with its simplest instance, that is as solutions of polynomial equations. We only use elementary notions about affine varieties and almost no background in algebraic geometry is required.

Let $k$ be a field and $\mathcal{I}$ be an ideal of $k[X_1,\dots, X_p]$. We define
$$
V(I) := \{x \in k^p~|~P(x) = 0~\forall P \in \mathcal{I}\}.
$$

\medskip

\begin{definition} Let $\theta \in \Phi^+$. Write $I_{\theta} := \llbracket 0, h(\theta^{\vee})-1 \rrbracket$. Notice that if $\theta$ is a simple root then $I_{\theta}=\{0\}$. For any root $\theta \in \Delta$ we set $P_{\theta} = X_{\theta}$ and $\lambda_{\theta}=0$. We denote by $P_{\theta}[\lambda_{\theta}]$ the polynomial $P_{\theta}+ \lambda_{\theta}-X_{\theta} \in \mathbb{Z}[X_{\Phi^+}]$. We define the ideal $J_{W_a}$ of $\mathbb{Z}[X_{\Phi^+}]$ as $${J_{W_a} := \sum\limits_{\alpha \in \Phi^+} ~\langle \prod\limits_{\lambda_{\alpha} \in I_{\alpha}} P_{\alpha}[\lambda_{\alpha}] \rangle}.$$ We define then the affine variety $X_{W_a}$ by
$$
  X_{W_a} := V(J_{W_a}).
$$
\end{definition}

\begin{remark}
The affine variety $X_{W_a}$ is seen as an affine variety over the field $\mathbb{R}$, but could be also defined over the field $\mathbb{C}$. We will see later on that it is an union of affine subspaces of $\mathbb{R}^m$. Therefore, each notion involved about the components is elementary.
\end{remark}

		\medskip

\begin{definition}\label{admissible vector}
 We say that $v=(v_{\alpha})_{\alpha \in \Phi^+} \in \mathbb{N}^m$ is an \emph{admissible vector} (or just \emph{admissible}) if it satisfies the boundary conditions, that is if for all $\alpha \in \Phi^+$ one has $v_{\alpha} \in I_{\alpha}$. For instance, all the $\lambda:=(\lambda_{\alpha})_{\alpha \in \Phi^+}$ coming from the polynomials $P_{\alpha}[\lambda_{\alpha}]$ give rise to admissible vectors. Furthermore, each admissible vector arises this way. For short we will write $\lambda$ instead of $(\lambda_{\alpha})_{\alpha \in \Phi^+}$.
\end{definition}
 
 		\medskip		
 \begin{definition}
   Let $\lambda$ be an admissible vector. We denote 
  $$
  J_{W_a}[\lambda] := \sum\limits_{\alpha \in \Phi^+} \langle P_{\alpha}[\lambda_{\alpha}] \rangle = \langle P_{\alpha}[\lambda_{\alpha}],~\alpha \in \Phi^+\rangle,
 $$ 
		and						
$$
X_{W_a}[\lambda] := V(J_{W_a}[\lambda])~~\text{and}~~X_{W_a}[0] := V(J_{W_a}[0_{\mathbb{R}^m}]).
$$
 \end{definition}

		\medskip

\begin{remark}
We recall that our goal is to have a bijection between $W_a$ and the integral points of a variety. $X_{W_a}$ seems to be a good candidate because the way we defined it shows that $W_a \hookrightarrow X_{W_a}(\mathbb{Z})$. However, it turns out that $X_{W_a}$ is too large. We explain this later on.
\end{remark}

		\medskip		
\begin{example}
Let us take $W_a =W(\widetilde{B}_2)$ with simple system $\Delta = \{\alpha, \beta\}$ and $\alpha$ the short root of $\Delta$. The positive root system is $\Phi^+=\{ \alpha, \beta, \alpha + \beta, 2\alpha +  \beta \}$. Because of Example \ref{coeffB} we know that there exists $\lambda_{\alpha +\beta}(w) \in \llbracket 0,2 \rrbracket$ such that
$k(w,\alpha + \beta) = k(w,\alpha) + 2k(w,\beta) + \lambda_{\alpha +\beta}(w) $, and there exists $ \lambda_{2\alpha + \beta}(w) \in \llbracket 0, 1 \rrbracket$ such that $k(w,2\alpha + \beta) = k(w,\alpha)+ k(w,\beta) +  \lambda_{\alpha + 2 \beta}(w)$. Hence we have 6 choices of $\lambda=(\lambda_{\theta})_{\theta \in \Phi^+}$. However, if we set  $k(w,\alpha) = 0$ and $k(w,\beta) = 0$ we see in Figure \ref{alcoveB2} that there are only 4 alcoves satisfying this and not 6. Those are $(0,0,0,0), (0,0,1,0),(0,0,1,1)$ and $(0,0,2,1)$ with the reading direction of Figure \ref{senslectureB2}.

\begin{figure}[h!]
\centering
\includegraphics[scale=0.2]{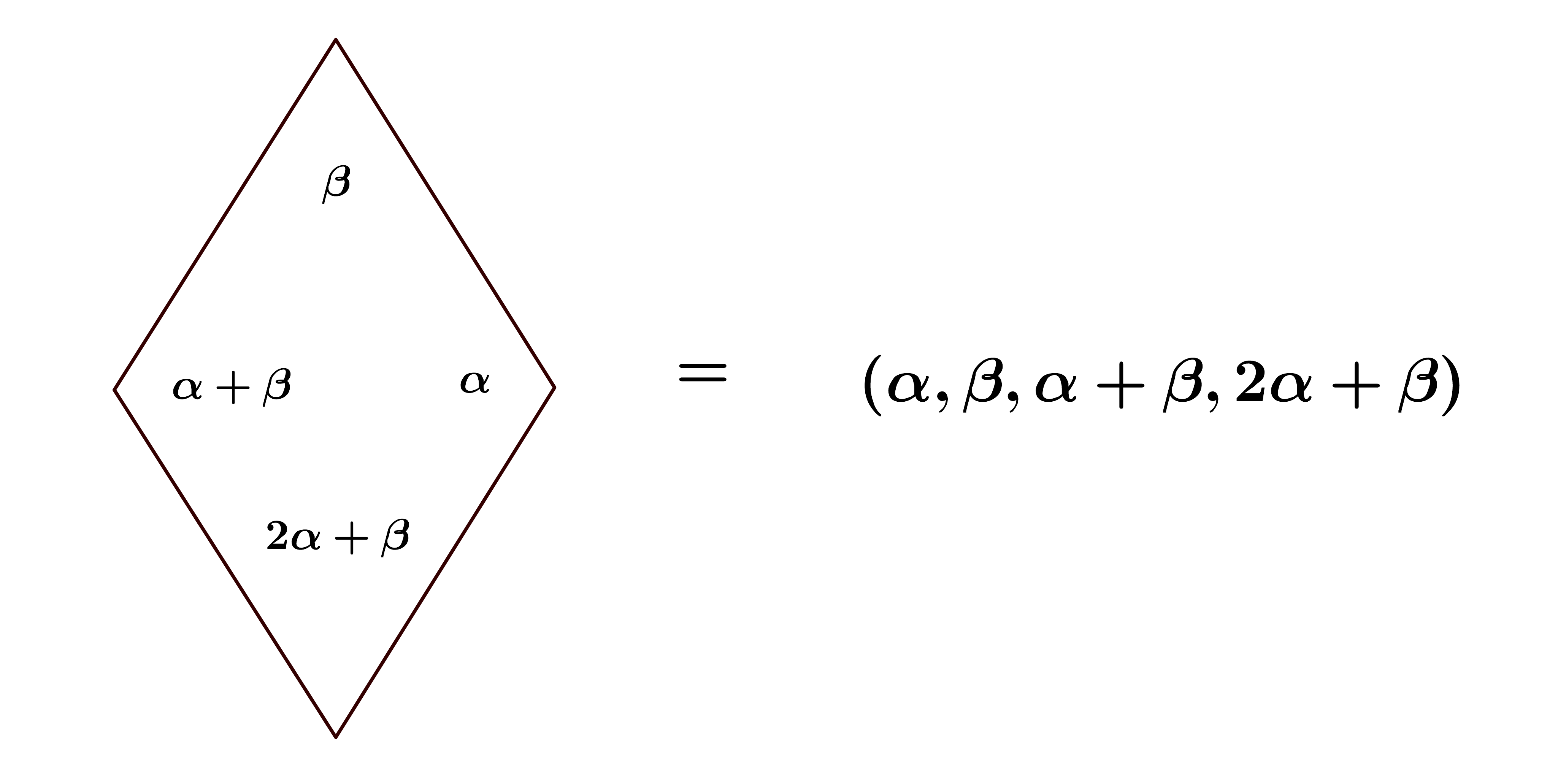} 
\caption{Reading direction of a $\Phi^+$-tuple in $W(\widetilde{B}_2)$.}
\label{senslectureB2}
\end{figure}

\begin{figure}[h!]
\centering
\includegraphics[scale=0.35]{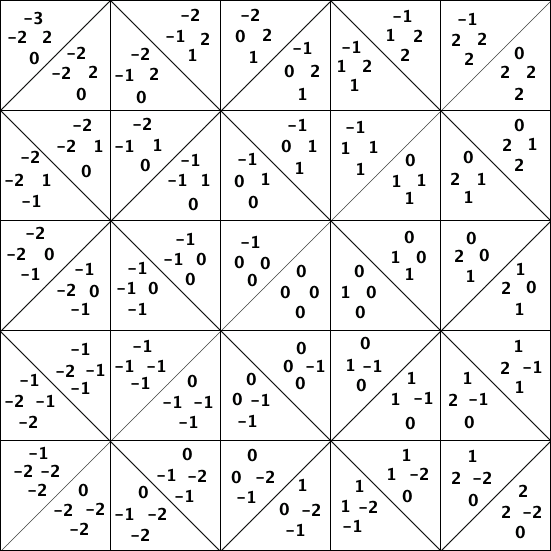} 
\caption{Alcoves in $W(\widetilde{B}_2)$.}
\label{alcoveB2}
\end{figure}
\end{example}

\newpage

\subsection{Admissible, admitted vectors and components of $\widehat{X}_{W_a}$ }

\begin{definition}\label{admitted vector}
We will denote by $S[W_a]$ the system of all the inequalities coming from Theorem \ref{thJYS1} or equivalently Theorem $\ref{thJYS3}$. Let $\lambda$ be an admissible vector. We say that $\lambda$ is \emph{admitted} if it satisfies the system $S[W_a]$.
\end{definition}

\begin{proposition}
Let $w \in W_a$. Let us write $k(w,\theta) = P_{\theta}(w) + \lambda_{\theta}(w)$ for all $\theta \in \Phi^+$. Then $(k(w,\theta))_{\theta \in \Phi^+}$ satisfies $S[W_a]$ if and only if $(\lambda_{\theta}(w))_{\theta \in \Phi^+}$ satisfies $S[W_a]$.
\end{proposition}

\begin{proof}
Let $\theta, \alpha, \beta \in \Phi^+$ such that  $\theta^{\vee} = \alpha^{\vee} + \beta^{\vee}$. Hence by Theorem \ref{thJYS3} one has $k(w,\alpha) + k(w,\beta) \leq k(w,\theta)\leq k(w,\alpha) + k(w,\beta) + 1$. Moreover by Theorem \ref{polynome} we know that $k(w,\alpha) = P_{\alpha}(w) + \lambda_{\alpha}(w)$ and $k(w,\beta) = P_{\beta}(w) + \lambda_{\beta}(w)$. We also know by Theorem \ref{polynome} that $P_{\theta} = P_{\alpha} + P_{\beta}$. Therefore we get $P_{\theta}(w) = P_{\alpha}(w) + P_{\beta}(w)$ and the previous inequalities become
$$
P_{\alpha}(w) + \lambda_{\alpha}(w) + P_{\beta}(w) + \lambda_{\beta}(w)\leq P_{\theta}(w) + \lambda_{\theta}(w) \leq P_{\alpha}(w) + \lambda_{\alpha}(w) + P_{\beta}(w) + \lambda_{\beta}(w) +1. 
$$
These are the same as 
$$
\lambda_{\alpha}(w) + \lambda_{\beta}(w)\leq  \lambda_{\theta}(w) \leq  \lambda_{\alpha}(w)  + \lambda_{\beta}(w) +1. 
$$
Thus  $(k(w,\alpha))_{\alpha \in \Phi^+} \in S[W_a]$ if and only if $(\lambda_{\alpha})_{\alpha \in \Phi^+}  \in S[W_a]$.
\end{proof}

\begin{itemize}
\item[]
\end{itemize}

The following proposition gives the irreducible components of $X_{W_a}$ together with their shape. We will strongly use this decomposition further.

\begin{proposition}\label{simplification}
$X_{W_a}$ decomposes as $X_{W_a} = \bigsqcup\limits_{\lambda~\text{admissible}} X_{W_a}[\lambda]$. The irreducible components of $X_{W_a}$ are the $X_{W_a}[\lambda]$ for $\lambda$ admissible and we have  ${X_{W_a}[\lambda]={X_{W_a}[0]+ \lambda}}$. Moreover, each irreducible component is of dimension $|\Delta|=n$. 
\end{proposition}

\begin{proof}
From the way we defined 	$X_{W_a}$ it is obvious that $X_{W_a} = \bigsqcup\limits_{\lambda~\text{admissible}} X_{W_a}[\lambda]$. Moreover, being an element of $X_{W_a}[\lambda]$, for $\lambda$ admissible, is the same as being a solution of the system 
$$
\{ P_{\alpha}+\lambda_{\alpha}-X_{\alpha}=0~|~\alpha \in \Phi^+ \}.
$$

But the solutions of this system are exactly the solutions of the system 
$$
\{ P_{\alpha}-X_{\alpha}=0~|~\alpha \in \Phi^+ \} + (\lambda_{\alpha})_{\alpha \in \Phi^+},
$$
that is the solutions of the system $X_{W_a}[0]+ \lambda$. Since $X_{W_a}[0]+ \lambda$ is an affine space it is clear that it is an irreducible component of $X_{W_a}$. As $X_{W_a}[\lambda] := V(J_{W_a}[\lambda])$ we see that this is just the intersection of the hyperplanes associated to the polynomials $P_{\alpha}[\lambda_{\alpha}]$  for all $\alpha \in \Phi^+$. Since there are $|\Phi^+|-|\Delta|$ such polynomials and since they are linearly independent as linear forms on $\mathbb R^{|\Phi^+|}$, it follows that the dimension of $X_{W_a}[\lambda]$ is $|\Delta|$.
\end{proof}

\begin{notation}
Let $Y \subset \mathbb{R}^m$. We denote by $Y(\mathbb{Z})$ the set of integral points of $Y$. 
\end{notation}

\begin{theorem}\label{theorem admitted}
Let $\gamma$ be an admissible vector. The following points are equivalent
\begin{itemize}
\item[i)] $\gamma$ is admitted.
\item[ii)] There is $x \in X_{W_a}[\gamma]$ such that there exists $w \in W_a$ satisfying $x = (k(w,\alpha))_{\alpha \in \Phi^+}$.
\item[iii)] All the integral points of $X_{W_a}[\gamma]$ arise as in ii).
\end{itemize}
\end{theorem}

\begin{proof}
From Theorem \ref{thJYS3} we know that a $\Phi^+$-tuple of integers $(k(w,\alpha))_{\alpha \in \Phi^+}$ associated to an element $w \in W_a$ is entirely characterized by the system of inequalities $S[W_a]$. We also know that a $\Phi^+$-tuple $(k_{\alpha})_{\alpha \in \Phi^+} \in \mathbb{Z}^{m}$ satisfying the system $S[W_a]$ provides an element ${w \in W_a}$ such that $(k(w,\alpha))_{\alpha \in \Phi^+} = (k_{\alpha})_{\alpha \in \Phi^+}$.

Let $x=(x_{\alpha})_{\alpha \in \Phi^+} \in X_{W_a}(\mathbb{Z})[\gamma]$. We can write $x_{\alpha}$ as $x_{\alpha} = P_{\alpha}(x) + \gamma_{\alpha}$. Then, thanks to Proposition \ref{simplification} if $\gamma$ is admitted we have $x_{\alpha}$ that satisfies $S[W_a]$. However, satisfying $S[W_a]$ implies being a $\Phi^+$-tuple of an element of $W_a$. Thus, there exists $w \in W_a$ such that ${x = (k(w, \alpha))_{\alpha \in \Phi^+}}$. The direction $i) \Rightarrow ii)$ follows.

Conversely, if $x \in X_{W_a}[\gamma]$ such that $x=(k(w,\alpha))_{\alpha \in \Phi^+}$ for some $w \in W_a$, then we have again $k(w,\alpha) = P_{\alpha}(w) + \gamma_{\alpha}$ for all $\alpha \in \Phi^+$. Since $w \in W_a$, $(k(w, \alpha))_{\alpha \in \Phi^+}$ must satisfy $S[W_a]$, but once again because of Proposition \ref{simplification} this is equivalent to say that $\gamma = (\gamma_{\alpha})_{\alpha \in \Phi^+}$ is also a solution of $S[W_a]$, that is $\gamma$ is admitted. The direction $ii) \Rightarrow i)$ follows. 

If we have now a point $x'=(x'_{\alpha})_{\alpha \in \Phi^+} \in X_{W_a}(\mathbb{Z})[\gamma]$ with $x' \neq x$ then $x'$ decomposes as ${x'_{\alpha}=P_{\alpha}(x')+\gamma_{\alpha}}$ for all $\alpha \in \Phi^+$. However, since $x$ is coming from an element of $W_a$ we know that $\gamma$ is admitted, which means that it is a solution of $S[W_a]$. By using Proposition \ref{simplification} it follows that $(x'_{\alpha})_{\alpha \in \Phi^+}$ satisfies $S[W_a]$. Thus, $x'$ is coming from a point of $W_a$, that is there exists $w' \in W_a$ such that $x'=(k(w', \alpha))_{\alpha \in \Phi^+}$. This concludes the case $ii) \Rightarrow iii)$.
The direction $iii) \Rightarrow ii)$ is clear. 
\end{proof}

\begin{theorem}\label{TH central}
\text{1)} The map $\iota : {W_a \rightarrow X_{W_a}(\mathbb{Z})}$ defined by $w \longmapsto (k(w,\alpha))_{\alpha \in \Phi^+}$ induces by corestriction a bijective map from $W_a$ to the integral points of a subvariety of $X_{W_a}$, denoted $\widehat{X}_{W_a}$, which we call the Shi variety of $W_a$. This subvariety is nothing but $ \widehat{X}_{W_a} := \bigsqcup\limits_{\lambda~\text{admitted}} X_{W_a}[\lambda]$. In other words, one has the following diagram:

\begin{center}
\begin{tikzpicture} 

\node at (0, 0) {$W_a$} ;
\node at (4, 0) {$X_{W_a}(\mathbb{Z}	)$} ;
\node at (4, -2) {$\widehat{X}_{W_a}(\mathbb{Z}).$} ;
\node at (3, -0.8) {$\circlearrowleft$} ;
%\node at (4.15, -0.9) {$\iota$} ;
\node at (2, 0.2) {$\iota$} ; 
\node at (2,-0.99) {\rotatebox{-125}{$\wr$}} ;

\usetikzlibrary{arrows}
\draw [right hook->] (0.3,0) -- (3,0) ;
\draw [dashed, ->] (0.3,-0.2) -- (3,-1.7) ;
\draw [right hook ->] (4,-1.6) -- (4,-0.22) ; 

\end{tikzpicture}

\end{center}
\text{2)} Let $\widetilde{\alpha}= \sum\limits_{i=1}^nc_i\alpha_i$ be the highest  root in $\Phi^+$. The number of irreducible components of $\widehat{X}_{W_a}$ is $n!\prod\limits_{i=1}^nc_i$ .
\end{theorem}

\begin{proof}
1) From the way we defined $X_{W_a}$ we know that each element of $W_a$, seen as $\Phi^+$-tuple, belongs to $X_{W_a}(\mathbb{Z})$. That is $\iota$ is well defined. Moreover, we also know because of Proposition 5.1 in \cite{JYS1} that $\iota$ is injective. Via Theorem \ref{theorem admitted}, one has $\iota(w) \in X_{W_a}[\lambda]$ if and only if $\lambda$ is admitted. Therefore, by deleting all the components $X_{W_a}[\gamma]$ with $\gamma$ admissible but non-admitted, $\iota$ becomes bijective.

2) The number of irreducible components of $\widehat{X}_{W_a}$ is equal to the number of admitted vectors. Moreover, by Theorem \ref{theorem admitted} we know that for each admitted vector  $\gamma = (\gamma_{\alpha})_{\alpha \in \Phi^+}$ there exists a unique $w \in W_a$ such that $\gamma_{\alpha} =k(w,\alpha)$ for all $\alpha \in \Phi^+$.  Since $\gamma_{\alpha} = 0$ for all $\alpha \in \Delta$, we have a bijection between the admitted vectors and the alcoves that belong to ${P_{\mathcal{H}} = \bigcap\limits_{\alpha \in \Delta}H_{\alpha,0}^1}$. Thanks to Section \ref{P} we know that 
$|\text{Alc}(P_{\mathcal{H}^{\vee}})|  = \frac{|W({\Delta})|}{f_{\Phi}}$. Therefore, by duality it follows that $|\text{Alc}(P_{\mathcal{H}})|  = \frac{|W({\Delta^{\vee}})|}{f_{\Phi^{\vee}}}$. However it is clear that $|W(\Phi^{\vee})| = |W(\Phi)|$, and since $f_{\Phi} = f_{\Phi^{\vee}}$ (see Section \ref{P}) it follows that $|\text{Alc}(P_{\mathcal{H}})|  = \frac{|W({\Delta})|}{f_{\Phi}}$. However the number $|W({\Delta})|$ is known and is equal to $f_{\Phi}n!\prod\limits_{i=1}^nc_i$ (see for example \cite{BOURB}, Ch. VI, $\S$ 2, prop. 7). The result follows.
\end{proof}

\medskip

\begin{notation}\label{notation compo vector}
For $\iota(w) \in \widehat{X}_{W_a}[\lambda]$ we denote $\lambda(w) := \lambda$. For instance $X_{W_a}[\lambda(s)]$ is the irreducible component having the $\Phi^+$-tuple $(k(s, \alpha))_{\alpha \in \Phi^+}$ with $s \in S$.
\end{notation}

\medskip

\begin{definition}
We define $H^0(\widehat{X}_{W_a})$ to be the set of irreducible components of $\widehat{X}_{W_a}$. For $\lambda=(\lambda_{\alpha})_{\alpha \in \Phi^+}$ an admitted vector we denote by $w_{\lambda}$ the associated element of $\text{Alc}(P_{\mathcal{H}})$, that is $w_{\lambda}$ is such that $k(w_{\lambda}, \alpha) = \lambda_{\alpha}$ for all $\alpha \in \Phi^+$. We give in the next table the number of irreducible components of any affine Weyl group.
\end{definition}

\begin{table}[h!]\label{TI}
\caption{Number of irreducible components of the Shi variety in any type.}
\begin{center}
\begin{tabular}{|c|c|c|c|c|}
\hline
 Type       &       Coefficients of $\widetilde{\alpha}$       &       Index of connection  &    $|W|$   &  $|H^0(\widehat{X}_{W_a})|$  \\
 \hline
 $A_n$     &   1,1,...,1                       								   &      $n+1 $                           &    $(n+1)! $          &      $n!$            \\
 \hline
$B_n$      &   1,2,2,...,2															&      2									&	  $2^{n}n!$     &             $2^{n-1}n!$        \\
\hline
$C_n$     &    2,2,...,2,1 													    &      2									&    $2^{n}n!$     &   $2^{n-1}n!$        \\
\hline
$D_n$     &    1,2,...,2,1,1 														&      4									&    $2^{n-1}n!$      &  $2^{n-3}n!$        \\
\hline
$E_6$      &   1,2,2,3,2,1														&      3									&	   $2^73^45$     &       $2^73^35$               \\ 
\hline
$E_7$      &   2,2,3,4,3,2,1														&      2									&    $2^{10}3^45\cdot7$     &    $2^{9}3^45\cdot7$            \\
\hline
$E_8$      &   2,3,4,6,5,4,3,2													&      1									&    $2^{14}3^55^27$     &        $2^{14}3^55^27$       \\
\hline
$F_4$      &   2,3,4,2																&      1 									&	  $2^73^2$     &     $2^73^2$              \\
\hline
$G_2$     &   3,2																	&      1									&    $2^23$     &  $2^23$                   \\
\hline
\end{tabular}
\end{center}
\end{table}

\begin{example}
Let us take the group $W(\widetilde{A}_2)$. The variety $\widehat{X}_{W(\widetilde{A}_2)}$ has two components which are drawn in Figure \ref{component A2} either in orange or in white. Thus, we have the decomposition as follows:
$$
\widehat{X}_{W(\widetilde{A}_2)} = X_{W(\widetilde{A}_2)}[(0,0,0)] \sqcup X_{W(\widetilde{A}_2)}[(0,0,1)].
$$

\vspace*{\stretch{1}}
\begin{figure}[h!]
\centering
\includegraphics[scale=0.2]{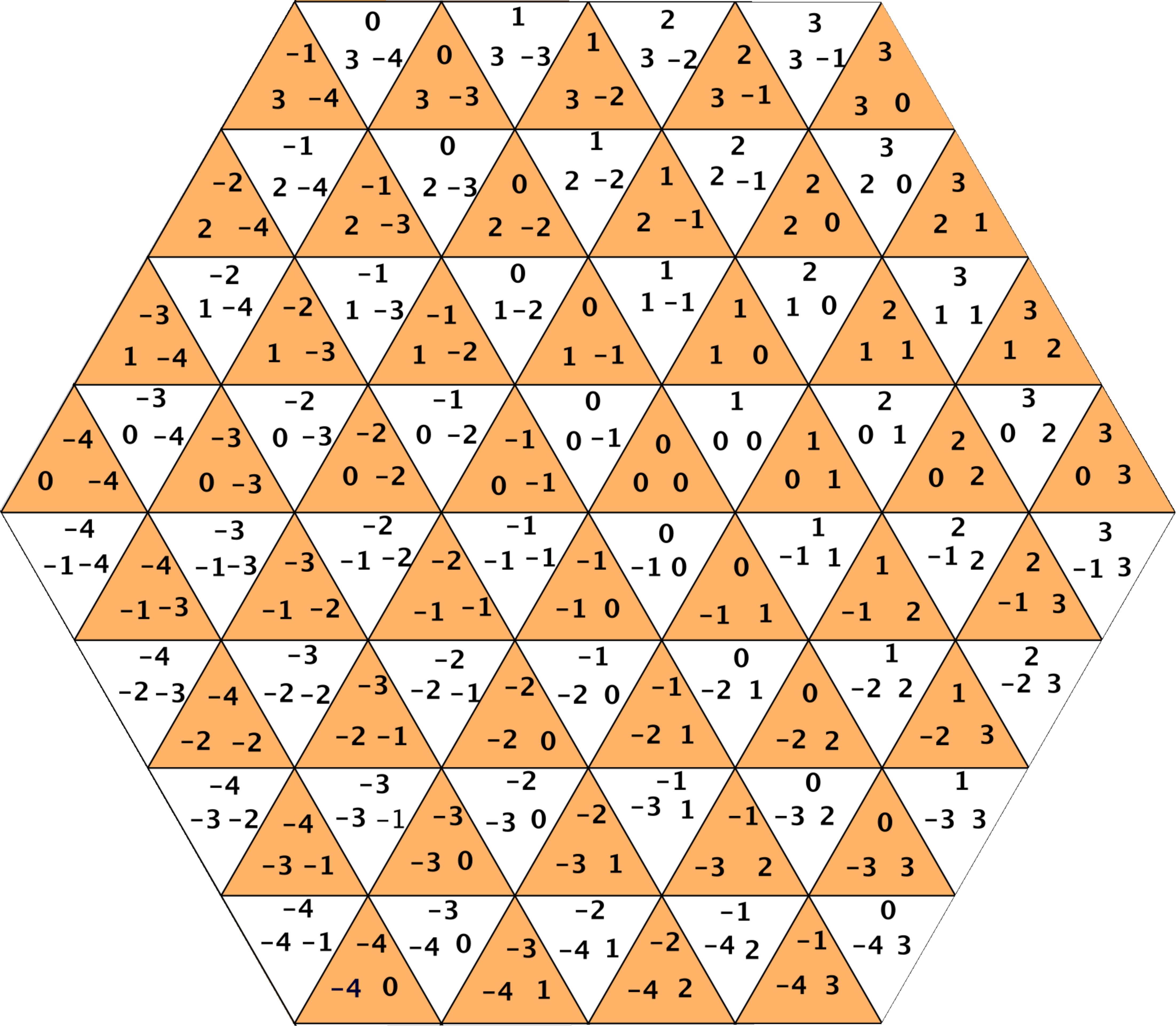} 
\caption{Irreducible components of $\widehat{X}_{W(\widetilde{A}_2)}$. }
\label{component A2}
\end{figure}
\vspace*{\stretch{1}}

\end{example}

\begin{example}
Let us take the group $W(\widetilde{B}_2)$. In this group there are 4 components, given by the 4 colors below. Indeed one has the following splitting according to the admitted vectors
  \begin{align*}
  \widehat{X}_{W(\widetilde{B}_2)} & =  X_{W(\widetilde{B}_2)}[(0,0,0,0)] ~\sqcup ~X_{W(\widetilde{B}_2)}[(0,0,1,0)] ~\sqcup  \\
  													& ~~\text{~} \text{~} \text{~}X_{W(\widetilde{B}_2)}[(0,0,1,1)] ~\sqcup~ X_{W(\widetilde{B}_2)}[(0,0,2,1)].
  \end{align*}

The first component corresponds with the pink, the second one with the yellow, the third one with the blue and the last one with the white. Moreover we can see that all the information is contained in the parallelogram where there are the 4 alcoves associated to the 4 admitted vectors. We also see that any reflection sends a component to another one. We will see in Proposition \ref{action isom} that it is not a coincidence. It is also possible to see that the lattice leaves stable the components. For example if we pick up the alcove $A_w=(0,0,0,0)$ and if we take $x=e_1+e_2$ then $A_{\tau_xw} = (0,2,2,2)$ and they are both of the same color pink.

\bigskip

\begin{figure}[h!]
\centering
\includegraphics[scale=0.35]{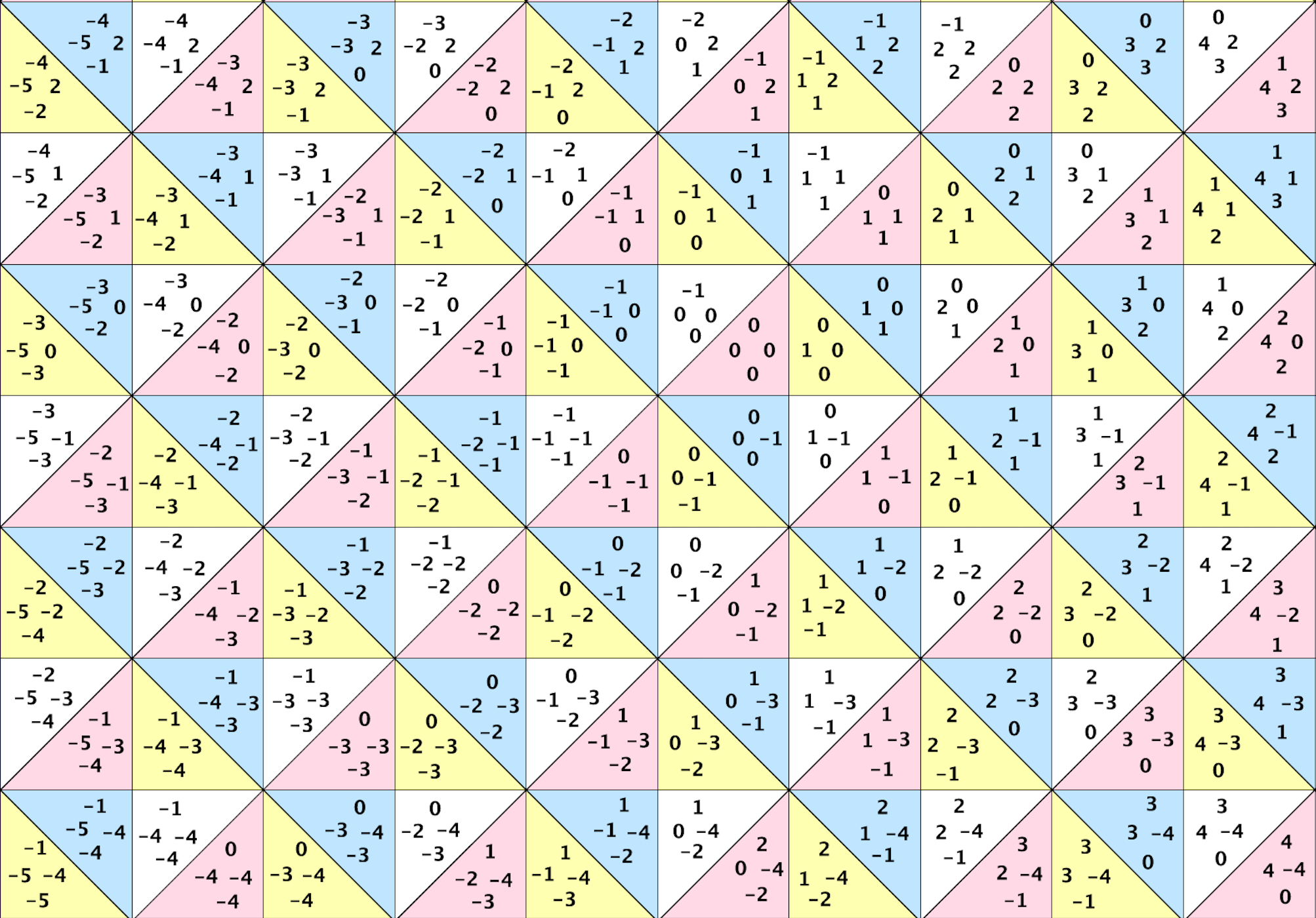} 
\caption{Irreducible components of $\widehat{X}_{W(\widetilde{B}_2)}$.}
\label{component B2}
\end{figure}
\end{example}

\bigskip

\begin{example}
Let us take the group $W(\widetilde{G}_2)$. We can see all the irreducible components in the polytope $P_{G_2}$. Here for example we have the 12 irreducible components of $W(\widetilde{G}_2)$. If we would color any element of the finite part $W(G_2)$ according to the components in which they are, we would see that all the colors appear in $W(G_2)$. The reason is that $f_{G_2} = 1$.

\begin{figure}[h!]
\centering
\includegraphics[scale=0.35]{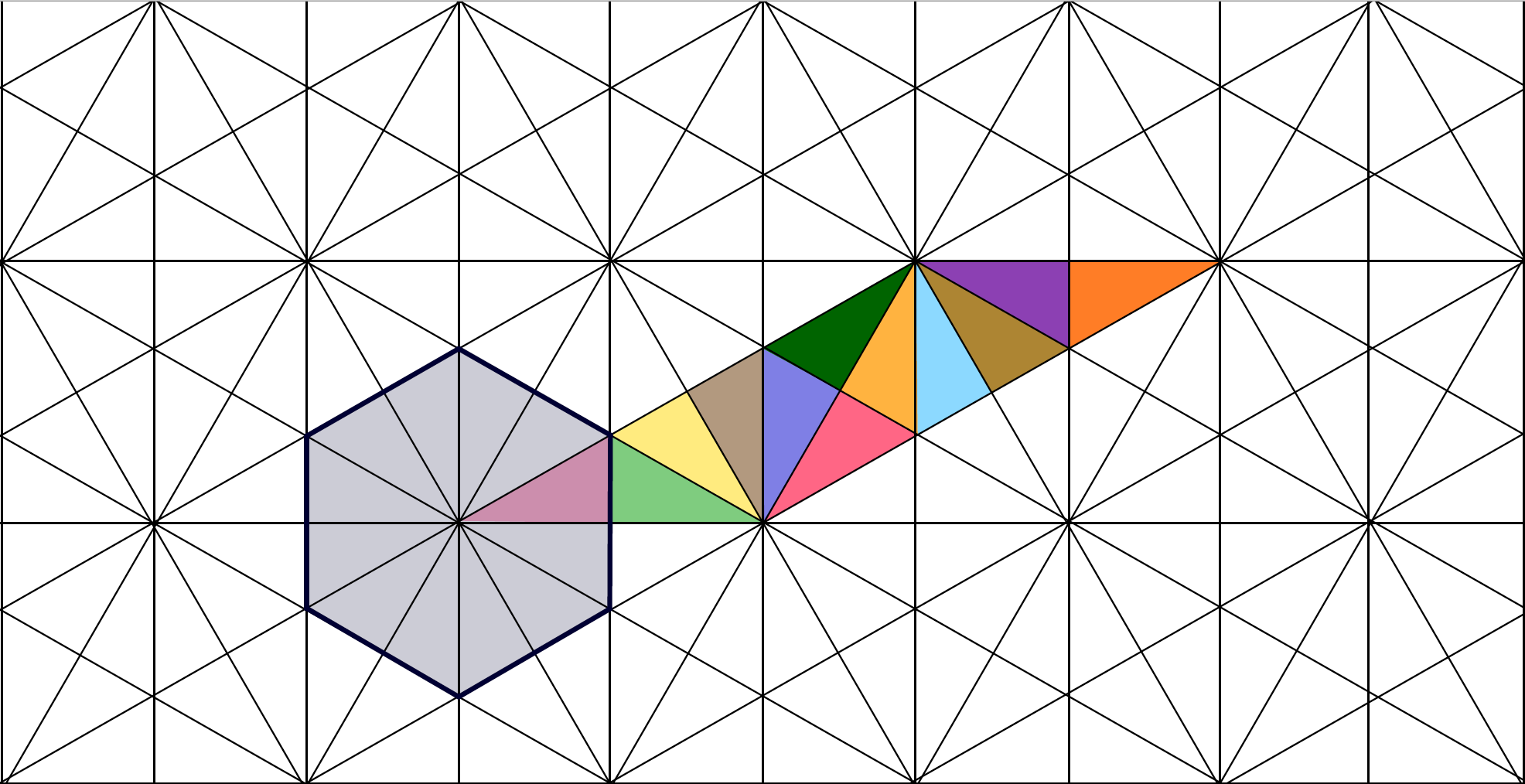} 
\caption{Alcoves of $P_{G_2}$ seen as set of representatives of irreducible components of $\widehat{X}_{W(\widetilde{G}_2)}$.}
\label{PG2}
\end{figure}
\end{example}

\subsection{Action on the components}
In this section we investigate the geometrical action associated to the $\Phi^+$-representation that we define in Proposition \ref{action isom}. We show in Proposition \ref{orbites} that this action is related to the index of connection of $\Phi$. The last result of this section shows where the Coxeter generators of $W$ lie in the Shi variety.

\begin{lemma}\label{decomposition equations}
Let $\theta$ be in $\Phi^+$. Then $\theta^{\vee} = P_{\theta}(\alpha_1^{\vee},\dots ,\alpha_n^{\vee})$.
\end{lemma}

\begin{proof}
Let us write $\theta^{\vee} = c_1\alpha_1^{\vee} +\cdots+ c_n\alpha_n^{\vee}$ with $c_i \in \mathbb{N}$. Then there exists $i \in \llbracket 1,n \rrbracket$ such that $\mu^{\vee}:= c_1\alpha_1^{\vee} + \cdots + (c_{i}-1)\alpha_i^{\vee} +\cdots+c_n\alpha_n^{\vee} \in (\Phi^{\vee})^+$. Thus, by using Theorem \ref{polynome} we have $P_{\theta} = P_{\mu} + P_{\alpha_i}$. Proceeding by induction the same way on $\mu^{\vee}$  we get that ${P_{\theta} = c_1P_{\alpha_1} + \cdots + c_nP_{\alpha_n}}$. However $P_{\alpha_i} = X_{\alpha_i}$ for all $i \in \llbracket 1,n \rrbracket$. It follows that ${P_{\theta} = c_1X_{\alpha_1}+\cdots+c_nX_{\alpha_n}}$.
\end{proof}

\begin{lemma}\label{translation lemma}
Let $w=\tau_x\overline{w}$ be an element of $W_a$ such that $x \in \mathbb{Z}\Phi$ and $\overline{w} \in W$. Let $\lambda$ be an admitted vector. Assume that $(k(w,\alpha))_{\alpha \in \Phi^+} \in X_{W_a}[\lambda]$. Then $(k(\overline{w},\alpha))_{\alpha \in \Phi^+} \in X_{W_a}[\lambda]$. In particular $X_{W_a}[\lambda]$ is stable under the action of $\mathbb{Z}\Phi$.
\end{lemma}

\begin{proof}
We just have to show that for any $\alpha \in \Phi^+$ one has $\lambda_{\alpha}(w) = \lambda_{\alpha}(\overline{w})$. We know that $k(w,\alpha) = P_{\alpha}(w) + \lambda_{\alpha}(w)$. For all  $\varepsilon \in \Delta$, thanks to Lemma \ref{proposition translation}, we have 
\begin{align*}
k(w, \varepsilon) &= k(\overline{w}, \varepsilon) + (x, \varepsilon^{\vee} ) \\
&= P_{\alpha}(\overline{w})+ \lambda_{\alpha}(\overline{w})+( x, \varepsilon^{\vee} ).
\end{align*}

However, we also have that 
\begin{align*}
k(w,\alpha)&= P_{\alpha}(w) + \lambda_{\alpha}(w)\\
&= P_{\alpha}(\{k(\overline{w},\varepsilon) +  ( x, \varepsilon^{\vee} ) \}_{\varepsilon \in \Delta}) + \lambda_{\alpha}(w).
\end{align*}

Since $P_{\alpha}$ is a linear polynomial we can split up the sum inside the variables and we get 
\begin{align*}
k(w, \alpha) & = P_{\alpha}(\{k(\overline{w},\varepsilon)\}_{\varepsilon \in \Delta})  + P_{\alpha}(\{( x, \varepsilon^{\vee} ) \}_{\varepsilon \in \Delta})+\lambda_{\alpha}(w) \\
                  & = k(\overline{w}, \alpha) - \lambda_{\alpha}(\overline{w}) +  P_{\alpha}(\{( x, \varepsilon^{\vee} ) \}_{\varepsilon \in \Delta})+\lambda_{\alpha}(w)\\
                  & = k(\overline{w}, \alpha) - \lambda_{\alpha}(\overline{w}) + ( x, P_{\alpha}(\{\varepsilon^{\vee}\}_{\varepsilon \in \Delta}) ) + \lambda_{\alpha}(w) \\
                  & = k(\overline{w}, \alpha) - \lambda_{\alpha}(\overline{w}) + ( x,\alpha^{\vee} ) + \lambda_{\alpha}(w). \\
\end{align*}
Finally it follows that $k(\overline{w}, \alpha) + ( x, \alpha^{\vee} ) = k(\overline{w}, \alpha) - \lambda_{\alpha}(\overline{w}) + ( x,\alpha^{\vee} ) + \lambda_{\alpha}(w)$ which is the same as $\lambda_{\alpha}(w) = \lambda_{\alpha}(\overline{w})$. 

For the second statement let us take $g \in W_a$ such that $(k(g,\alpha))_{\alpha \in \Phi^+} \in X_{W_a}[\lambda]$ and let us take $y \in \mathbb{Z}\Phi$.  What we have done just before implies that $(k(\overline{\tau_yg}, \alpha))_{\alpha \in \Phi^+} $ and $(k(\tau_yg, \alpha))_{\alpha \in \Phi^+}$ belong to the same component. Since $\overline{\tau_yg} = \overline{g}$ we get that $(k(\overline{g},\alpha))_{\alpha \in \Phi^+}$ and $(k(\tau_yg, \alpha))_{\alpha \in \Phi^+}$ are in the same component. But once again we know that $(k(\overline{g},\alpha))_{\alpha \in \Phi^+}$ and $(k(g,\alpha))_{\alpha \in \Phi^+}$ are in the same component.  It follows that $(k(g,\alpha))_{\alpha \in \Phi^+}$ and $(k(\tau_yg,\alpha))_{\alpha \in \Phi^+}$ are in the same component. Hence we have shown that $X_{W_a}[\lambda](\mathbb{Z})$ is stable under the action of $\mathbb{Z}\Phi$. It follows that $X_{W_a}[\lambda]$ is also stable under $\mathbb{Z}\Phi$.
\end{proof}

\begin{proposition}\label{action isom}
Let $F : W_a \hookrightarrow$ Isom$(\mathbb{R}^n)$ be the $\Phi^+$-representation of $W_a$. Then 
\begin{itemize}
\item[1)] $W_a$ acts naturally on the irreducible components of $\widehat{X}_{W_a}$ via the action defined as ${w\diamond X_{W_a}[\lambda] := F(w)(X_{W_a}[\lambda])}$ for $\lambda$ admitted. Furthermore if we assume that $w\in W_a$ decomposes as ${w=\tau_x\overline{w}}$,  then ${w \diamond X_{W_a}[\lambda]=\overline{w}\diamond X_{W_a}[\lambda]}$. Finally this action is transitive. 
\item[2)] The previous action induces an action on the admitted vectors by $w\diamond \lambda := \gamma$ such that $w\diamond X_{W_a}[\lambda] = X_{W_a}[\gamma]$. In other words we have $w\diamond X[\lambda]=X[w \diamond \lambda]$.
\end{itemize}
\end{proposition}

\begin{proof}

1) Let $x \in W_a$ such that $(k(x,\alpha))_{\alpha \in \Phi^+} \in X_{W_a}[\lambda]$. For any $w \in W_a$ there exists an admitted vector  $\gamma$ such that $(k(wx,\alpha))_{\alpha \in \Phi^+} \in X_{W_a}[\gamma]$. However, since we know that $F(w)(x) = (k(wx,\alpha))_{\alpha \in \Phi^+}$, it follows that $F(w)$ sends at least one element of  $X_{W_a}[\lambda]$ to $X_{W_a}[\gamma]$. Moreover $F(w)$ is an isometry and then $F(w)(X_{W_a}[\lambda])$ must be an affine space of the same dimension as  $X_{W_a}[\lambda]$. 

Let us show now that $(F(w)(X_{W_a}[\lambda]))(\mathbb{Z}) = F(w)(X_{W_a}[\lambda](\mathbb{Z}))$. First of all the inclusion $  F(w)(X_{W_a}[\lambda](\mathbb{Z})) \subset (F(w)(X_{W_a}[\lambda]))(\mathbb{Z})$ is clear. With respect to the other inclusion, let us take $z \in (F(w)(X_{W_a}[\lambda]))(\mathbb{Z})$. Then we have $F(w)^{-1}(z) \in X_{W_a}[\lambda]$ but $F(w)^{-1}=F(w^{-1})$ and the coefficients of the matrix associated to $F(w^{-1})$ are integers together with the coordinates of the translation part. It follows that $F(w)^{-1}(z) \in X_{W_a}[\lambda](\mathbb{Z})$ and thus ${z \in F(w)(X_{W_a}[\lambda](\mathbb{Z}))}$. The equality follows.

Finally, $F(w)(X_{W_a}[\lambda])$ is an affine space of dimension $|\Delta|$ all of whose integer points are associated to elements of $W_a$. It follows that $F(w)(X_{W_a}[\lambda])$ must be an irreducible component of $\widehat{X}_{W_a}$. However  $(k(wx,\alpha))_{\alpha \in \Phi^+} \in X_{W_a}[\gamma]$, then since the components of $X_{W_a}$ have no intersection we must have  ${F(w)(X_{W_a}[\lambda])=X_{W_a}[\gamma]}$. Thus, we have shown that the following map is well defined 
$$
\begin{array}{ccc}
W_a \times H^0(\widehat{X}_{W_a}) & \longrightarrow & H^0(\widehat{X}_{W_a}) \\
(w, X_{W_a}[\lambda]) & \longmapsto & w \diamond X_{W_a}[\lambda].
\end{array}
$$

Since $F$ is a morphism we have 
$$
e \diamond X_{W_a}[\lambda] = F(e)(X_{W_a}[\lambda]) = id(X_{W_a}[\lambda]) = X_{W_a}[\lambda],
$$ and for all $w, w' \in W_a$ we also have 
\begin{align*}
ww' \diamond X_{W_a}[\lambda] &= F(ww')(X_{W_a}[\lambda]) \\ &= F(w)F(w')(X_{W_a}[\lambda]) = F(w)(F(w')(X_{W_a}[\lambda])) \\
													& = w\diamond (w'\diamond X_{W_a}[\lambda]).
\end{align*}

This concludes the first statement of 1). With respect to the second statement we just have to see that the action of $F(\tau_x)$ doesn't do anything on the components. By Lemma \ref{translation lemma} we know that $X_{W_a}[\lambda]$ is stable under $\mathbb{Z}\Phi$. It follows that $F(\tau_x)(X_{W_a}[\lambda]) = X_{W_a}[\lambda]$.

To show the transitivity we just have to see that we can obtain any component of $H^0(\widehat{X}_{W_a})$ from the component $X_{W_a}[0]$. Let $X_{W_a}[\gamma]$ be such a component and $(k(g,\alpha))_{\alpha \in \Phi^+} \in X_{W_a}[\gamma]$. It follows that $(k(\overline{g},\alpha))_{\alpha \in \Phi^+} \in X_{W_a}[\gamma]$. As $\overline{g}=\overline{g}e$ and since $(k(e,\alpha))_{\alpha \in \Phi^+} \in X_{W_a}[0]$ it follows that $F(\overline{g})(X_{W_a}[0]) = X_{W_a}[\gamma]$, that is $\overline{g}\diamond X_{W_a}[0]=X_{W_a}[\gamma]$.
\item[2)]This is a straightforward consequence of the first point.
\end{proof}

\begin{proposition}\label{orbites}
Let $f_{\Phi}$ be the index of connection of $\Phi$ and let $\lambda$ be an admitted vector. 
\begin{itemize}
\item[1)] In W there are exactly $f_{\Phi}$ elements belonging to  $X_{W_a}[\lambda]$.
\item[2)] $X_{W_a}[\lambda]$ has exactly $f_{\Phi}$ orbits under the action of $\mathbb{Z}\Phi$.
\item[3)] Each orbit of $X_{W_a}[\lambda]$ has only one element of $W$.
\end{itemize}
\end{proposition}

\begin{proof}
1) First of all we note that $W \cap \iota^{-1}(X_{W_a}[\lambda](\mathbb{Z})) \neq \emptyset$. Indeed, for any $w=\tau_x\overline{w} \in \iota^{-1}(X_{W_a}[\lambda](\mathbb{Z}))$ we know by Lemma \ref{translation lemma} that $(k(\overline{w},\alpha))_{\alpha \in \Phi^+}$ is also in  $X_{W_a}[\lambda]$, that is $\overline{w} \in W \cap \iota^{-1}(X_{W_a}[\lambda](\mathbb{Z}))$. We need now to show that if we take another admitted vector $\gamma$ then $|W \cap \iota^{-1}(X_{W_a}[\lambda](\mathbb{Z}))| =|W \cap \iota^{-1}(X_{W_a}[\gamma](\mathbb{Z}))| $. But this is true because any component can be sent to any other one via an element of $W$. In particular there exists $g \in W$ such that ${g\diamond X_{W_a}[\lambda] = X_{W_a}[\gamma]}$. Further, if $(k(w,\alpha))_{\alpha \in \phi^+} \in X_{W_a}[\lambda]$ with $w \in W$, then $F(g)((k(w,\alpha))_{\alpha \in \phi^+})$ is also associated to an element of $W$. The result follows since $g$ is a bijection. By Theorem \ref{TH central} we know that there are $n!\prod\limits_{i=1}^nc_i$ irreducible components in $\widehat{X}_{W_a}$ where $\widetilde{\alpha}= \sum\limits_{i=1}^nc_i\alpha_i$ is the longest root in $\Phi^+$. However it is well known 
\begin{comment} (see \cite{RichardKane}  ch III Section 11-6 or \cite{Hum} section 4.9)
\end{comment}
 that $|W| = n!(\prod\limits_{i=1}^nc_i)f_{\Phi}$. Since all the irreducible components have same number of elements in $W$, we deduce that this number is the quotient of $|W|$ by the number of irreducible components, which is exactly $f_{\Phi}$.

2) This is a consequence of the point above. Let us take the $f_{\Phi}$ elements $\{ w_1,\dots, w_{f_{\Phi}} \}$ of ${W \cap \iota^{-1}(X_{W_a}[\lambda](\mathbb{Z}))}$. Then the $\mathbb{Z}\Phi$-orbits $O_{w_i}$ and $O_{w_j}$ are different for all $i,j \in \{1,\dots ,f_{\Phi}\}$. Indeed, if there was an element belonging to both then it would follow that there exists $x \in \mathbb{Z}\Phi$ such that $w_i=\tau_xw_j$, but this is impossible because $w_i$ and $w_j$ are in the finite group $W$. We therefore have $f_{\Phi}$ distinct orbits. We claim that there is no other orbit. Indeed, let $x \in X_{W_a}[\lambda]$. Then there exists $w\in W_a$ such that $x=\iota(w)$. Moreover, one can write $w=\tau_y\overline{w}$ with $y \in \mathbb{Z}\Phi$ and $\overline{w} \in W$. This implies that $\iota(w)$ and $\iota(\overline{w})$ are in the same orbit $O_x$. However, since $\overline{w} \in W \cap \iota^{-1}(X_{W_a}[\lambda](\mathbb{Z}))$ it follows that $O_x = O_{w_i}$ for some $i=1,\dots ,f_{\Phi}$.

3) Assume that we have an orbit in $X_{W_a}[\lambda]$ with at least two elements $w_1$, $w_2$ coming from $W$. Then there exists $x \in \mathbb{Z}\Phi$ such that $w_2=\tau_xw_1$. Since $w_1$ and $w_2 \in W$ the previous equality cannot occur. It follows that each orbit has at most one element of $W$. Moreover, by the first point above we know that $X_{W_a}[\lambda]$ has $f_{\phi}$ elements of $W$. Since there are $f_{\phi}$ orbits in $X_{W_a}[\lambda]$ it follows that each orbit has a unique element of $W$.
\end{proof}

\medskip

\begin{proposition}\label{prop position generateur}
Let $W_a$ be an affine Weyl group such that $W_a \neq W(\widetilde{A}_2)$. Let $s_i, s_j \in S$ such that $s_i \neq s_j$. Then $\widehat{X}_{W_a}[\lambda(s_i)] \neq  \widehat{X}_{W_a}[\lambda(s_j)]$ (where the meaning of $X_{W_a}[\lambda(s_i)]$ is defined in Notation \ref{notation compo vector}).
\end{proposition}

\begin{proof}
For the $B_2$ ($=C_2$) and $G_2$ cases one can directly see the result on Figures \ref{component B2} and \ref{PG2}. We assume now that the rank of $\Phi$ is greater or equal than 3. Let us assume that $s_i$ and $s_j$ are in the same component $\widehat{X}_{W_a}[\lambda(s_i)]$. We claim that there exists a root $\theta \in \Phi^+\setminus \Delta$ such that $P_{\theta}(s_i) \neq P_{\theta}(s_j)$. Assuming this claim, because of Theorem \ref{polynome} and Corollary \ref{corollaire formule 3.3.1} we have $k(s_i,\theta) = P_{\theta}(s_i) + \lambda_{\theta}(s_i)=0$ and ${k(s_j,\theta) = P_{\theta}(s_j) + \lambda_{\theta}(s_j)=0}$, which implies that $\lambda_{\theta}(s_i) \neq \lambda_{\theta}(s_j)$. However, since $s_i$ and $s_j $ are in $\widehat{X}_{W_a}[\lambda(s_i)]$, we must have $\lambda(s_i) = \lambda(s_j)$ and then  $\lambda_{\alpha}(s_i) = \lambda_{\alpha}(s_j)$ for all $\alpha \in \Phi^+\setminus \Delta$. Therefore, $s_i$ and $s_j$ cannot lie in the same component.

\textit{Proof of the claim}.  Let $\Phi$ be a root system satisfying rank$(\Phi) \geq 3$. This is a general fact, which can be checked case by case, that in any such a crystallographic root system there always exists a root $\theta \in \Phi^+\setminus \Delta$ such that either $\alpha_i$ is in the expression of $\theta$ (according to $\Delta$) but $\alpha_j$ is not, or $\alpha_j$ is in the expression of $\theta$ (according to $\Delta$) but $\alpha_i$ is not. We apply this statement on the dual root system of $\Phi$.

Let $\theta \in \Phi^+\setminus \Delta$ such that $\alpha_i^{\vee}$ is in the expression of $\theta^{\vee}$ (according to $\Delta^{\vee}$) but $\alpha_j^{\vee}$ is not, the other case being symmetric. Thus, because of Lemma \ref{decomposition equations} we have $X_{\alpha_i}$ in the expression of $P_{\theta}$, whereas $X_{\alpha_j}$ doesn't appear in this expression. Again because of Lemma \ref{decomposition equations} we know that all the coefficients of the polynomial $P_{\theta}$ are positive. Therefore, since $k(s_i, \alpha) = 0$ for all $\alpha \in \Phi^+\setminus \{\alpha_i\}$ and $k(s_i,\alpha_i) = -1$ we have $P_{\theta}(s_i) \leq -1$. Moreover, we also have   $k(s_j, \alpha) = 0$ for all $\alpha \in \Phi^+\setminus \{\alpha_j\}$ and $k(s_j,\alpha_j) = -1$, which implies that $P_{\theta}(s_j) = 0$. 
\end{proof}

\medskip

\begin{remark}
The $A_2$ case fails because there are not enough roots. See Figure \ref{component A2} where the generators $s_0=(0,0,1)$, $s_1 = (-1,0,0)$, and $s_2=(0,-1,0)$ lie in the same component. The reading direction according to Figure \ref{dessin1} is as follows: $(\alpha, \beta, \alpha + \beta)$.
\end{remark}

\newpage

\textbf{Acknowledgements}. This work was initiated in the LACIM (Montréal) under the supervision of Christophe Hohlweg and benefitted from a lot of discussions with Antoine Abram, Robert Bédard, François Bergeron, Nicolas England, Christophe Reutenauer, Franco Saliola and Hugh Thomas. The author would like to thank Christophe Hohlweg and Hugh Thomas for answering many questions and for many helpful comments. The author is particularly grateful to Antoine Abram, Christophe Reutenauer and Hugh Thomas for valuable discussions and questions, which led to many interesting results. I also wish to thank Riccardo Biagioli and Matthew Dyer for some helpful comments and communications.

This work was partially supported by NSERC grants and by the LACIM.

\begin{itemize}
\item[]
\end{itemize}

   \noindent Nathan Chapelier-Laget \\
   D\'epartement de math\'ematiques \\
   Universit\'e du Qu\'ebec \`a Montr\'eal \\
   \email{nathan.chapelier@gmail.com}

\nocite{*}
\bibliographystyle{plain}
\bibliography{shi_variety.bib}

\end{document}